\documentclass[11pt,reqno]{amsart}
\usepackage{cmap}
\usepackage[T1]{fontenc}
\usepackage[utf8]{inputenc}
\usepackage[english]{babel}
\usepackage{amsfonts,amssymb,amsthm,amsmath}
\usepackage{url}
\usepackage{enumitem}

\usepackage{secdot}

\allowdisplaybreaks

\usepackage{graphicx}
\usepackage{epstopdf}
\usepackage{subfig}

\usepackage{a4wide}
\usepackage{xcolor}
\usepackage[colorlinks=true,linktocpage,pdfpagelabels,
bookmarksnumbered,bookmarksopen]{hyperref}
\definecolor{ForestGreen}{rgb}{0.1,0.6,0.05}
\definecolor{EgyptBlue}{rgb}{0.063,0.1,0.6}
\hypersetup{
	colorlinks=true,
	linkcolor=EgyptBlue,         
	citecolor=ForestGreen,
	urlcolor=olive
}

\usepackage[hyperpageref]{backref}

\newtheorem{theorem}{Theorem}
\newtheorem{proposition}[theorem]{Proposition}
\newtheorem{lemma}[theorem]{Lemma}
\newtheorem{corollary}[theorem]{Corollary}

\theoremstyle{definition}
\newtheorem{definition}[theorem]{Definition}
\newtheorem{remark}[theorem]{Remark}

\let\OLDthebibliography\thebibliography
\renewcommand\thebibliography[1]{
	\OLDthebibliography{#1}
	\setlength{\parskip}{1pt}
	\setlength{\itemsep}{1pt plus 0.3ex}
}

\numberwithin{equation}{section}
\numberwithin{theorem}{section}
\numberwithin{equation}{section}
\numberwithin{theorem}{section}

\newcommand{\N}{\mathbb{N}}
\newcommand{\R}{\mathbb{R}}

\title[Basisness and completeness of Fu\v{c}\'ik eigenfunctions]{Basisness and completeness of Fu\v{c}\'ik eigenfunctions for the Neumann Laplacian}

\author[F.~Baustian]{Falko Baustian\,$^\ast$}
\author[V.~Bobkov]{Vladimir Bobkov}

\address[F.~Baustian]{\newline\indent
	Institute of Mathematics, University of Rostock,
	\newline\indent 
	Ulmenstra{\ss}e 69, 18057 Rostock, Germany
}
\email{falko.baustian@uni-rostock.de}

\address[V.~Bobkov]{\newline\indent
	Institute of Mathematics, Ufa Federal Research Centre, RAS,
	\newline\indent 
	Chernyshevsky str. 112, 450008 Ufa, Russia
}
\email{bobkov@matem.anrb.ru}

\date{}


\thanks{
V.~Bobkov was supported in the framework of executing the development program of Volga Region Mathematical Center (agreement no.~075-02-2022-888).
This work is supported by the German-Russian Interdisciplinary Science Center (G-RISC) funded by the German Federal Foreign Office via the German Academic Exchange Service (DAAD), Project F-2021b-8\_d.
}

\thanks{
$^\ast$ Corresponding Author.
}

\begin{document}
	\begin{abstract} 
	    We investigate the basis properties of sequences of Fu\v{c}\'ik eigenfunctions of the one-dimensional Neumann Laplacian. 
	    We show that any such sequence is complete in $L^2(0,\pi)$
	    and a Riesz basis in the subspace of functions with zero mean. 	       
	    Moreover, we provide sufficient assumptions on Fu\v{c}\'ik eigenvalues which guarantee that the corresponding Fu\v{c}\'ik eigenfunctions form a Riesz basis in $L^2(0,\pi)$ and we explicitly describe the corresponding biorthogonal system.
	\end{abstract}
	
	\maketitle 
	
	\section{Introduction}
	
    We consider the generalized eigenvalue problem
	\begin{equation}\label{eq:fucik}
		\left\{
		\begin{alignedat}{1}
			-u''(x) &= \alpha u^{+}(x)-\beta u^{-}(x) ~~ \mbox{in }(0,\pi), \\
			u'(0) &=u'(\pi)=0,
		\end{alignedat}
		\right.
	\end{equation}
	where $\alpha,\beta\in\mathbb{R}$ are spectral parameters and $u^\pm := \max\{\pm u,0\}$ are the positive and negative parts of $u$ so that $u=u^+ - u^-$.
	This problem is called the \textit{Fu\v{c}\'ik eigenvalue problem} and it was introduced under Dirichlet boundary conditions by Fu\v{c}\'ik \cite{fucik} and Dancer \cite{dancer} to investigate elliptic equations with ``jumping'' nonlinearities. 
	The prominence of such equations and problems of the type \eqref{eq:fucik} is justified by the fact that they model mechanical systems with asymmetric oscillations. 
	Examples of such systems vary from simple spring-mass systems with two independent springs, to complex models of suspension bridges, see, e.g., \cite{DHMN,gazzola,LMk} and references therein for a detailed discussion.
	Apart from the importance in mechanics, problems of the type \eqref{eq:fucik} are challenging from the  mathematical point of view and different aspects have been intensively investigated, see, e.g., a short list  \cite{arias,massa,MW,rynne} for several results and overviews on the Fu\v{c}\'ik eigenvalue problem \eqref{eq:fucik}.
	
    The aim of the present paper is to study basis properties of sequences of solutions to the problem \eqref{eq:fucik}. Namely, we want to know if such systems of so-called Fu\v{c}\'ik eigenfunctions are complete in the real Hilbert space $L^2(0,\pi)$ or even form a Riesz basis in this space. 
	To the best of our knowledge, basisness for systems of Fu\v{c}\'ik eigenfunctions has been studied only under Dirichlet boundary conditions, see our previous works \cite{BB,BB1}. 
	We emphasize that the considered problem is a modification of the form of the standard eigenvalue problem for the Laplacian which it is not in the scope of the classical spectral theorem. 
	In the case of the standard eigenvalue problems, we refer to \cite{kato,skal1} for an overview on the basis properties of eigenfunctions of more complex linear operators. 
	Moreover, basis properties of eigenfunctions under Neumann boundary conditions were also studied for nonlinear operators, e.g., the $p$-Laplacian in \cite{BM2016,edmunds1}.

	\subsection{\texorpdfstring{Fu\v{c}\'ik}{Fucik} eigenvalues and eigenfunctions}\label{sec:fucik}
	
	The \textit{Fu\v{c}\'ik spectrum} $\Sigma \subset \mathbb{R}^2$ of the Fu\v{c}\'ik eigenvalue problem \eqref{eq:fucik} denotes the set of all pairs of parameters $(\alpha,\beta)$ for which \eqref{eq:fucik} possesses a nonzero classical solution.
	Any element of the Fu\v{c}\'ik spectrum is called a \textit{Fu\v{c}\'ik eigenvalue}, and any corresponding nonzero solution is called a \textit{Fu\v{c}\'ik eigenfunction}.
	
	Let us describe the structure of the Fu\v{c}\'ik spectrum $\Sigma$ and the corresponding eigenfunctions, see, e.g., \cite{exner}.
	For $\alpha=\beta =: \mu$, the problem \eqref{eq:fucik} becomes the standard eigenvalue problem $-u''= \mu u$ with homogeneous Neumann boundary conditions. 
	Its sequence of eigenvalues $\mu_n=n^2$, $n\in\N_0=\N\cup\{0\}$, has a constant function $\phi_0$ and the cosine functions $\phi_n=\cos(nx)$ as corresponding eigenfunctions. 
	Hereinafter, we specify $\phi_0$ to be equal to $\frac{\sqrt{2}}{2}$ in order to guarantee that $\left\|\phi_n\right\|=\sqrt{\frac{\pi}2}$ for every $n \in \N_0$.
   Hence, the lines $\{0\}\times\R$ and $\R\times\{0\}$ correspond to the constant Fu\v{c}\'ik eigenfunctions and form the trivial part of $\Sigma$.
    Any nonconstant Fu\v{c}\'ik eigenfunction starts with a scaled cosine function at $x=0$ to its first consecutive zero, followed by a combination of alternating positive and negative ``bumps'', and ends at $x=\pi$ with the extrema of a scaled cosine function, see, e.g., Figure \ref{fig:eigenf}. 
	The positive bumps consist of translations and multiples of $\cos(\sqrt{\alpha}x)$ and have the length $l_1=\frac{\pi}{\sqrt{\alpha}}$, whereas negative bumps have the form of translations and multiples of $\cos(\sqrt{\beta}x)$ with the length $l_2=\frac{\pi}{\sqrt{\beta}}$.
	
	\begin{figure}[h!]
		\centering
		\subfloat[][\centering $n=3$, $\alpha=5$]
		{\includegraphics[width=0.39\linewidth]{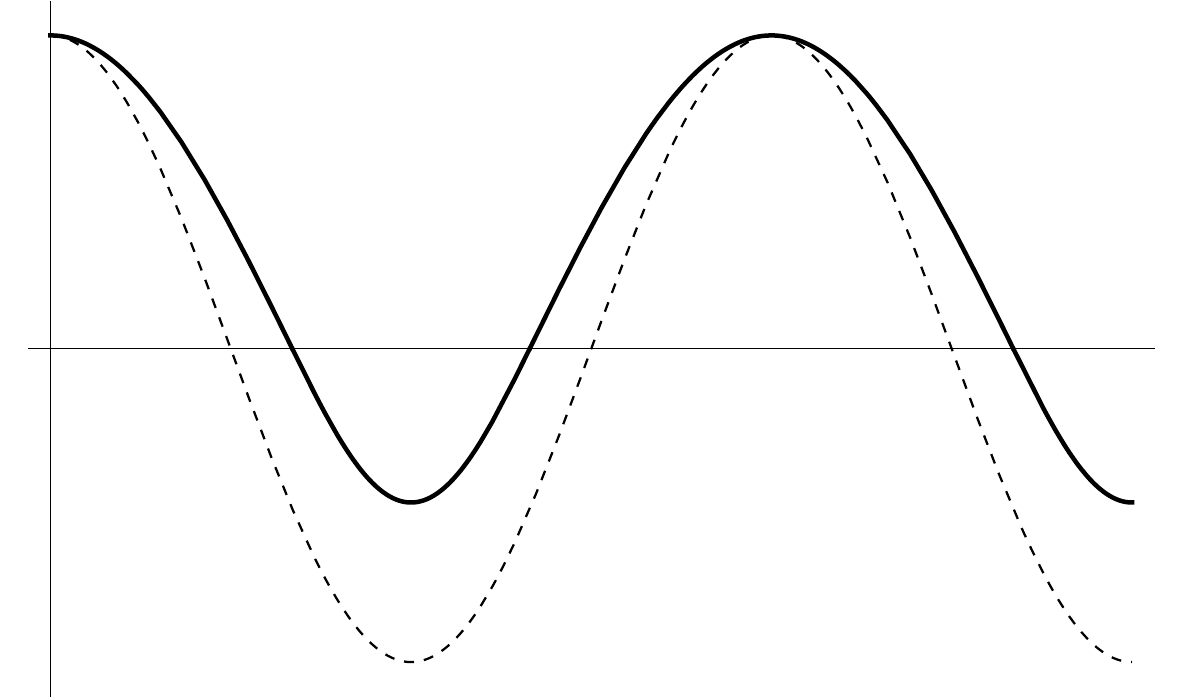}}
		\qquad\qquad
		\subfloat[][\centering $n=4$, $\alpha=60$]
		{\includegraphics[width=0.39\linewidth]{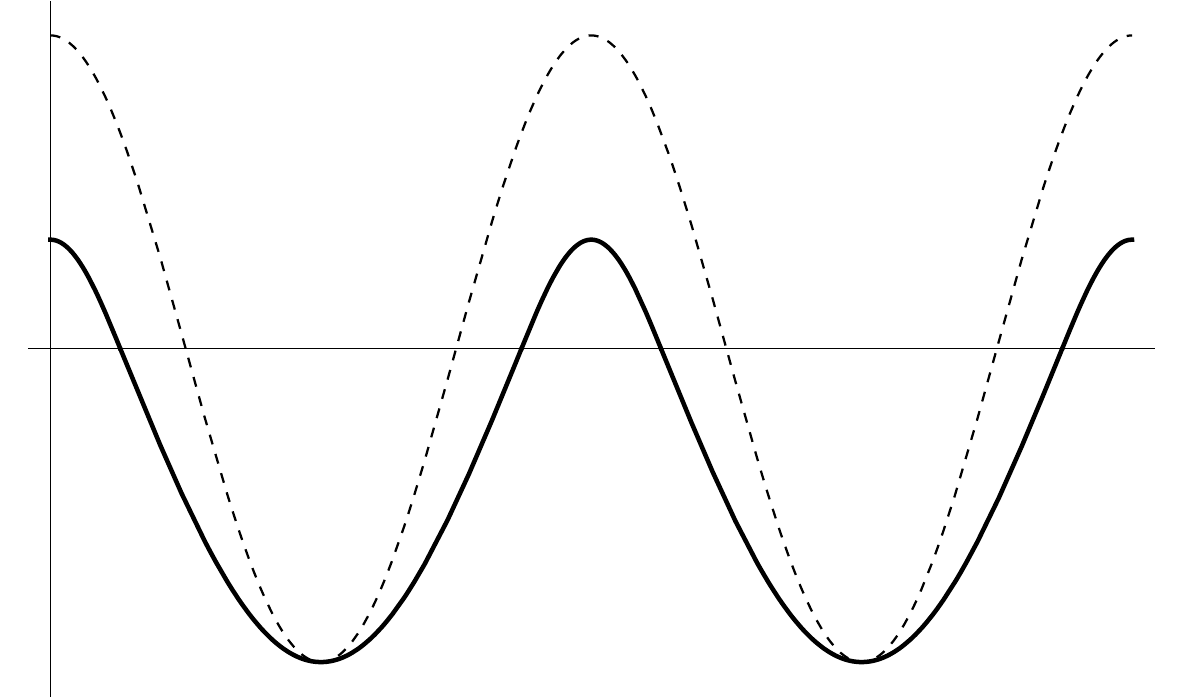}}
		\caption{A normalized Fu\v{c}\'ik eigenfunction $f^n_{\alpha,\beta}$ (solid) and $\cos(nx)$ (dashed)}
		\label{fig:eigenf}
	\end{figure}
	
    Therefore, the nontrivial part of $\Sigma$ is exhausted by the following hyperbola-type curves, see Figure \ref{fig1}:
	$$
	\Gamma_n
	=
	\left\{(\alpha,\beta)\in\R^2: \,
	\frac{n}{2}\frac{\pi}{\sqrt{\alpha}}
	+
	\frac{n}{2}\frac{\pi}{\sqrt{\beta}}
	=\pi
	\right\}, \quad n\in\N.
	$$
	
	\begin{figure}[h!]
		\center{\includegraphics[width=0.32\linewidth]{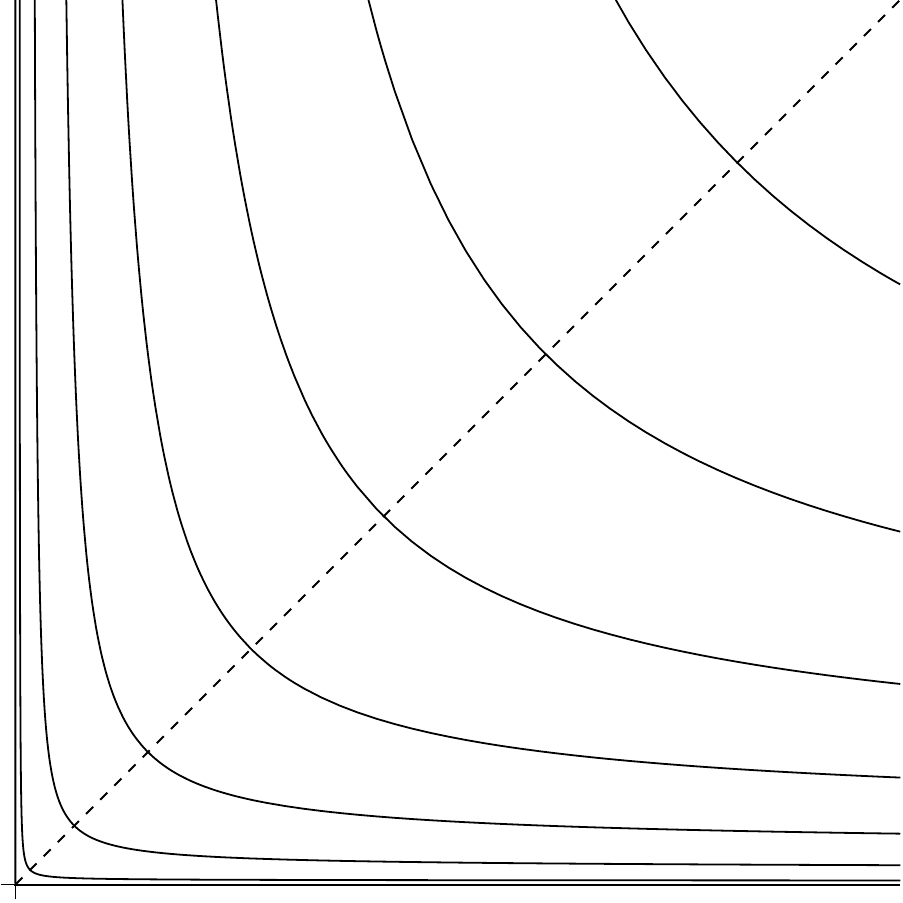}}
		\caption{Several curves of the Fu\v{c}\'ik spectrum}
		\label{fig1}
	\end{figure}
	
	\noindent
    Each $\Gamma_n$ contains the point $(\mu_n,\mu_n)$ corresponding to the 
    cosine eigenfunction $\phi_n$. Clearly, we have
	\begin{equation}\label{eq:ab}
	\alpha=\frac{n^2\beta}{(2\sqrt{\beta}-n)^2} \quad\mbox{and}\quad \beta=\frac{n^2\alpha}{(2\sqrt{\alpha}-n)^2}   
	\end{equation}
	for every Fu\v{c}\'ik eigenvalue $(\alpha,\beta)\in\Gamma_n$.	
	
	 Since each Fu\v{c}\'ik eigenvalue corresponds to multiple eigenfunctions, we uniquely specify a Fu\v{c}\'ik eigenfunction for every point of $\Sigma$.
	\begin{definition}
		Let $n \in \mathbb{N}$ and $(\alpha,\beta)\in \Gamma_n$. The \textit{normalized Fu\v{c}\'ik eigenfunction} $f^n_{\alpha,\beta}$
		is the $C^2$-solution of the problem \eqref{eq:fucik} with $f^n_{\alpha,\beta}(0)>0$ which is normalized by
		$$
		\|f^n_{\alpha,\beta}\|_{\infty}
		=
		\sup_{x\in[0,\pi]}|f^n_{\alpha,\beta}(x)|=1.
		$$
		Without loss of generality, we choose $f^0_{\alpha,\beta}=\phi_0 \equiv \frac{\sqrt{2}}{2}$ for every $(\alpha,\beta)\in(\{0\}\times\R)\cup(\R\times\{0\})$.
	\end{definition}
	
	The normalized Fu\v{c}\'ik eigenfunctions and their $2\pi$-periodic extensions on $\mathbb{R}$ can be characterized by the following piecewise definition.
	Let $n \in \N$.
	For $\alpha\geq n^2 \geq \beta$, we have
	\begin{equation}\label{eq:fucikpiecewise1}
		f^n_{\alpha,\beta}(x)=
		\left\{
		\begin{array}{clrl}
			\frac{\sqrt{\beta}}{\sqrt{\alpha}}\cos\left(\sqrt{\alpha}\,\left(x-\frac{2k}{2}l\right)\right) \quad &\mbox{for} 
			&\frac{2k}{2}l-\frac{l_1}{2}&\leq x<\frac{2k}{2}l+\frac{l_1}{2}, \\
			-\cos\left(\sqrt{\beta}\,\left(x-\frac{2k+1}{2}l\right)\right) \quad &\mbox{for}
			&\frac{2k+1}{2}l-\frac{l_2}{2}&\leq x<\frac{2k+1}{2}l+\frac{l_2}{2},
		\end{array}
		\right.
	\end{equation}
	and, for $\beta > n^2 > \alpha$, we have
	\begin{equation}\label{eq:fucikpiecewise2}
		f^n_{\alpha,\beta}(x)=
		\left\{
		\begin{array}{clrl}
			\cos\left(\sqrt{\alpha}\,\left(x-\frac{2k}{2}l\right)\right) \quad &\mbox{for} 
			&\frac{2k}{2}l-\frac{l_1}{2}&\leq x<\frac{2k}{2}l+\frac{l_1}{2}, \\
			-\frac{\sqrt{\alpha}}{\sqrt{\beta}}\cos\left(\sqrt{\beta}\,\left(x-\frac{2k+1}{2}l\right)\right) \quad &\mbox{for}
			&\frac{2k+1}{2}l-\frac{l_2}{2}&\leq x<\frac{2k+1}{2}l+\frac{l_2}{2},
		\end{array}
		\right.
	\end{equation}
	where $l=l_1+l_2$ and $k \in \mathbb{Z}$, cf.\ \cite[Note 2]{exner}.
	Notice that $f^n_{\alpha,\beta} \not\in C^3[0,\pi]$ provided $\alpha \neq \beta$.
	
	The symmetry of the curves $\Gamma_n$ of $\Sigma$ with respect to the diagonal $\alpha=\beta$ is related to certain properties of Fu\v{c}\'ik eigenfunctions which are not  presented in the case of the Dirichlet Laplacian.
	Namely, we make the following simple yet important observations.
	\begin{remark}\label{rem:sym}
		Since the distance between two extrema of a Fu\v{c}\'ik eigenfunction $f^n_{\alpha,\beta}$ is always  $\frac{l}{2}=\frac{l_1}{2}+\frac{l_2}{2}$, it must be equal to $\frac{\pi}{n}$ due to the form of $f^n_{\alpha,\beta}$. 
		Hence, the function $f^n_{\alpha,\beta}$ is $l$-periodic and, thus, $\frac{2\pi}{n}$-periodic as the cosine function $\phi_n$. 
		Furthermore, the extrema of Fu\v{c}\'ik eigenfunctions are located at the same points as the extrema of the corresponding cosine functions, see Figure \ref{fig:eigenf}.
		The piecewise definitions \eqref{eq:fucikpiecewise1} and \eqref{eq:fucikpiecewise2} also provide the relations $f^n_{\alpha,\beta}(x)=-f^n_{\beta,\alpha}\left(x-\frac{\pi}{n}\right)$ for every $n\in\N$ and $f^n_{\alpha,\beta}(x)=-f^n_{\beta,\alpha}\left(\pi-x\right)$ for each odd $n\in\N$. 
		These observations greatly simplify calculations in the proofs of our main results.
	\end{remark}
	
	We are interested in basis properties of sequences of Fu\v{c}\'ik eigenfunctions that we introduce in the following natural way.
	
	\begin{definition}\label{def:FS}
		We define a \textit{Fu\v{c}\'ik system} $F_{\alpha,\beta}=\{f^n_{\alpha(n),\beta(n)}\}_{n\in\mathbb{N}_0}$ as a sequence of normalized Fu\v{c}\'ik eigenfunctions with the mappings $\alpha,\beta\colon\N_0 \to\R$ satisfying $\alpha(0)=\beta(0)=0$ and $(\alpha(n),\beta(n))\in \Gamma_n$ for every $n\geq1$.
	\end{definition}
	
	This definition is consistent with the one from our previous work \cite{BB} on the basis properties of Fu\v{c}\'ik eigenfunctions for the Dirichlet Laplacian.

	\subsection{Main results}
	
	We now provide our main results on basis properties of Fu\v{c}\'ik systems. Let us mention that a Fu\v{c}\'ik system $F_{\alpha,\beta}$ is called complete in $L^2(0,\pi)$ if for every $f\in L^2(0,\pi)$ and every $\varepsilon>0$ there exist a sequence $\{c_n\}_{n\in\mathbb{N}_0}$ of constants and $N_0\in\mathbb{N}_0$ such that
	$$
	\Big\|f-\sum_{n=0}^{N} c_n f^n_{\alpha,\beta}\Big\|\leq\varepsilon
	\quad \mbox{for every }N\geq N_0.
	$$
	A Fu\v{c}\'ik system $F_{\alpha,\beta}$ is called a basis if it is complete and the sequence $\{c_n\}_{n\in\mathbb{N}_0}$ can be chosen independently of $\varepsilon$. A basis is called a Riesz basis if it is the image of a complete orthonormal system under some linear homeomorphism, see \cite{young}.
	
	One of the important tools in the study of basisness of sequences of functions is the concept of $\omega$-linear independence \cite{bary}. 
	A Fu\v{c}\'ik system $F_{\alpha,\beta}$ is called $\omega$-linearly independent if the strong convergence
	$$
	\sum_{n=0}^{\infty} \eta_nf^n_{\alpha,\beta}
	= 
	\lim_{k \to \infty} \Big\|\sum_{n=0}^{k}\eta_nf^n_{\alpha,\beta}\Big\|=0
	$$
	for any sequence of scalars $\{\eta_n\}_{n\in\mathbb{N}_0}$ implies $\eta_n=0$ for every $n\in\N_0$.
	Although the $\omega$-linear independence is usually a hardly verifiable property for general sequences of functions, we establish the following result.
	
	\begin{proposition}\label{thm:omega}
		Every Fu\v{c}\'ik system $F_{\alpha,\beta}$ is $\omega$-linearly independent.
	\end{proposition}
	
	Thanks to the $\omega$-linear independence, we obtain basisness for a wide class of Fu\v{c}\'{i}k systems.
	
	\begin{theorem}\label{thm:main1}
		Let $F_{\alpha,\beta}$ be a Fu\v{c}\'{i}k system.
		If the mappings $\alpha$ and $\beta$ satisfy
		\begin{equation}\label{eq:thm:main1}
			\sum_{n=1}^\infty 
			\frac{(\max(\sqrt{\alpha(n)},\sqrt{\beta(n)})-n)^2}{n^2}
			<
			\infty,
		\end{equation}
		then $F_{\alpha,\beta}$ is a Riesz basis in $L^2(0,\pi)$.
	\end{theorem}	
	
	The assumption \eqref{eq:thm:main1} is optimal in the sense of the method of the proof, see a discussion in Remark \ref{rem:conv} below.
	Estimating each term of the sum in \eqref{eq:thm:main1} by $C^2/n^{1+\varepsilon}$ for some fixed $C, \varepsilon > 0$, we derive the following result.
	\begin{corollary}
	Let $F_{\alpha,\beta}$ be a Fu\v{c}\'ik system. 
	Let $C,\varepsilon>0$ be fixed. 
	If the mappings $\alpha$ and $\beta$ satisfy
	\begin{equation*}		        \max(\sqrt{\alpha(n)},\sqrt{\beta(n)})
	\leq 
	n
	+ 
	C \, n^{(1-\varepsilon)/2}
	\end{equation*}
	for every $n \in \mathbb{N}$, 
	then $F_{\alpha,\beta}$ is a Riesz basis in $L^2(0,\pi)$.
\end{corollary}

    We observe that the assumption \eqref{eq:thm:main1} holds for each Fu\v{c}\'{i}k system $F_{\alpha,\beta}$ that has only a finite number of elements with $f^n_{\alpha,\beta}\neq\phi_n$, i.e., $\alpha(n)=n^2=\beta(n)$ for a.e.\ $n\in\mathbb{N}_0$. This is useful when we are interested in the approximation of a function $f\in L^2(0,\pi)$ by a finite sequence of arbitrary Fu\v{c}\'{i}k eigenfunctions $\{f^n_{\alpha,\beta}\}_{n=0}^{N}$, i.e.,
    \begin{equation}\label{eq:approx}
	    f \approx \sum_{n=0}^N c_nf^n_{\alpha,\beta}\,,
	\end{equation}
	for a sufficient large $N\in\mathbb{N}$. The Fu\v{c}\'{i}k system
	$$
	F_{\alpha,\beta}=\{f^n_{\alpha,\beta}\}_{n=0}^{N}\cup\{\phi_n\}_{n\geq N+1}
	$$
	satisfies \eqref{eq:thm:main1} and, hence, it is a Riesz basis in $L^2(0,\pi)$. We remark that every Fu\v{c}\'{i}k system $F_{\alpha,\beta}$ which is a Riesz basis in $L^2(0,\pi)$ possesses a unique complete biorthogonal sequence $\{\psi_n\}_{n\in\mathbb{N}_0}$ that we describe explicitly in Section \ref{sec:biort}. Thus, the coefficients of the approximation \eqref{eq:approx} are given by $c_n=\langle f,\psi_n \rangle$ for every $n\in\mathbb{N}_0$, where $\langle \cdot,\cdot \rangle$ stands for the standard scalar product in $L^2(0,\pi)$.

    Without requiring the assumption \eqref{eq:thm:main1}, we provide the following general result.
	\begin{theorem}\label{thm:main2}
		Let $F_{\alpha,\beta}$ be an arbitrary Fu\v{c}\'{i}k system.
		Then $F_{\alpha,\beta}$ is complete in $L^2(0,\pi)$.
		Moreover, the system
		$\{f^n_{\alpha,\beta}(x) - 
		\frac{\sqrt{2}}{\pi}\langle f^n_{\alpha,\beta},\phi_0\rangle\}_{n \in \mathbb{N}}$ is a Riesz basis in the space
		$$
		L^2_\perp(0,\pi)
		:=
		\{
		u \in L^2(0,\pi):~ \langle u,\phi_0\rangle=0
		\}.
		$$
	\end{theorem}	
	
	The rest of the article is organized as follows.
	In Section \ref{sec:2}, we provide expressions for the scalar products $\langle f^n_{\alpha,\beta},\phi_m\rangle$ for all combinations of $n,m \in \mathbb{N}_0$ and prove Proposition \ref{thm:omega}.
	In Section \ref{sec:dist}, we derive expressions for the distances $\|f^n_{\alpha,\beta}-\phi_n\|^2$ and prove Theorem \ref{thm:main1}.
	Section~\ref{sec:proof} is devoted to the proof of Theorem \ref{thm:main2} and also provides an alternative proof of the basisness of Fu\v{c}\'ik systems under the assumption \eqref{eq:thm:main1}.
	In Section \ref{sec:biort}, we characterize elements of biorthogonal sequences associated with Fu\v{c}\'ik systems.
	In Appendix \ref{sec:appendix}, we give a detailed proof of Proposition \ref{prop:scalar}.
    Finally, Appendix \ref{sec:appendix:rational} contains an auxiliary technical lemma which we employ in the proof of Theorem \ref{thm:main2}.

	\section{Scalar products and Proof of Proposition \ref{thm:omega}}\label{sec:2}
	In this section, we provide formulas for the scalar products of normalized Fu\v{c}\'ik eigenfunctions $f^n_{\alpha,\beta}$ and cosine functions $\phi_m$ and, using them, we prove the $\omega$-linear independence of arbitrary Fu\v{c}\'ik systems stated in Proposition \ref{thm:omega}.
	
	Let us recall the following notations:
	$$
	l_1 = \frac{\pi}{\sqrt{\alpha}},
    \qquad
    l_2 = \frac{\pi}{\sqrt{\beta}},
    \qquad
    l = l_1 + l_2 = \frac{2\pi}{n}.
	$$
	\begin{proposition}\label{prop:scalar}
		Let $m, n \in \mathbb{N}$ and $(\alpha,\beta) \in \Gamma_n$. 
		Then the following assertions hold true:
		\begin{enumerate}[label={\rm(\Roman*)}]
			\item\label{prop:scalar:I}
			Let $\alpha\geq n^2\geq\beta$.
			\begin{enumerate}[label={\rm(\roman*)}]
				\item\label{prop:scalar:I:i}
				If $m/n \in \mathbb{N}$ and $\alpha \neq m^2$, $\beta \neq m^2$, then 
				\begin{align}
					\notag
					\langle f^n_{\alpha,\beta},\phi_m\rangle
					&= (-1)^{m/n}\frac{(\beta-\alpha)\sqrt{\beta}\,n}{(m^2-\alpha)(m^2-\beta)}\cos\left(\frac{ml_2}{2}\right) \\
					\label{eq:scal:mdivn0}
					&= \frac{4\alpha^2n^2(n-\sqrt{\alpha})}{(m^2-\alpha)(m^2(2\sqrt{\alpha}-n)^2-n^2\alpha)(2\sqrt{\alpha}-n)}\cos\left(\frac{m\pi}{2\sqrt{\alpha}}\right).
				\end{align}
				\item\label{prop:scalar:I:ii}
				If $m/n \in \mathbb{N}$ and $\alpha = m^2$, $\beta \neq m^2$, then 
				\begin{equation}
					\label{eq:scal:aeqm-mdivn-main}
					\langle f^n_{\alpha,\beta},\phi_m\rangle
					= \frac{\sqrt{\beta}\pi n}{4m^2}
					=
					\frac{\pi n^2 \sqrt{\alpha}}{4m^2 (2\sqrt{\alpha}-n)}
					=
					\frac{\pi n^2}{4m (2m-n)}.
				\end{equation}
				\item\label{prop:scalar:I:iii}
				If $m/n \not\in \mathbb{N}$, then
				$\langle f^n_{\alpha,\beta},\phi_m\rangle = 0$.
			\end{enumerate}	
			\item\label{prop:scalar:II}
			Let $\beta> n^2>\alpha$.  
			\begin{enumerate}[label={\rm(\roman*)}]
				\item\label{prop:scalar:II:i}
				If $m/n \in \mathbb{N}$ and $\alpha \neq m^2$, $\beta \neq m^2$, then 
				\begin{align}
					\notag
					\langle f^n_{\alpha,\beta},\phi_m\rangle
					&= \frac{(\beta-\alpha)\sqrt{\alpha}\,n}{(m^2-\alpha)(m^2-\beta)}\cos\left(\frac{ml_1}{2}\right) \\
					\label{eq:scal:mdivn10}
					&= (-1)^{m/n+1}\frac{4\beta^2n^2(n-\sqrt{\beta})}{(m^2-\beta)(m^2(2\sqrt{\beta}-n)^2-n^2\beta)(2\sqrt{\beta}-n)}\cos\left(\frac{m\pi}{2\sqrt{\beta}}\right). 
				\end{align}
				\item\label{prop:scalar:II:ii}
				If $m/n \in \mathbb{N}$ and $\alpha \neq m^2$, $\beta = m^2$, then 
				\begin{equation}
					\label{aeqm-mdivn-main2}
					\langle f^n_{\alpha,\beta},\phi_m\rangle
					= (-1)^{m+1}\frac{\sqrt{\alpha}\pi n}{4m^2}
					=
					(-1)^{m+1}
					\frac{\pi n^2 \sqrt{\beta}}{4m^2 (2\sqrt{\beta}-n)}
					=
					(-1)^{m+1}
					\frac{\pi n^2}{4m (2m-n)}.
				\end{equation}
				\item\label{prop:scalar:II:iii}
				If $m/n \not\in \mathbb{N}$, then
				$\langle f^n_{\alpha,\beta},\phi_m\rangle = 0$.
			\end{enumerate}	
		\end{enumerate}
	\end{proposition}
	
	The proof of Proposition \ref{prop:scalar} is technical and presented in Appendix \ref{sec:appendix}. The proposition exhausts all nontrivial combinations of $m, n \in \mathbb{N}$ and $(\alpha,\beta) \in \Gamma_n$. 
	The only remaining case is $\alpha = \beta = m^2$,  which is possible if and only if $m=n$.
	In this case, we have  $f^n_{\alpha,\beta} = \phi_n$, and hence $\langle f^n_{\alpha,\beta},\phi_m\rangle = \frac{\pi}{2}$. 
	We want to explicitly mention that for any $n \in \mathbb{N}$,
	\begin{align}
		\label{eq:scalar-m0-1}
		\langle f^n_{\alpha,\beta},\phi_0\rangle 
		&=
		-\frac{2\sqrt{2}(\sqrt{\alpha}-n)}{2\sqrt{\alpha}-n}
		=
		\frac{2\sqrt{2}(\sqrt{\beta}-n)}{n}
		\quad \text{if}~ \alpha\geq n^2\geq\beta,\\
		\label{eq:scalar-m0-2}
		\langle f^n_{\alpha,\beta},\phi_0\rangle 
		&=
		-\frac{2\sqrt{2}(\sqrt{\alpha}-n)}{n}
		=
		\frac{2\sqrt{2}(\sqrt{\beta}-n)}{2\sqrt{\beta}-n}
		\quad\, \text{if}~ \beta> n^2>\alpha.
	\end{align}
	Moreover, $\langle f^0_{\alpha,\beta},\phi_0\rangle = \frac{\pi}{2}$ and $\langle f^0_{\alpha,\beta},\phi_m\rangle = 0$ for every $m \in \mathbb{N}$, thanks to the fact that $f^0_{\alpha,\beta} = \phi_0$.

	\begin{remark}\label{rem:nn}
		The expression for the scalar product $\langle f^n_{\alpha,\beta},\phi_n\rangle$ can be also derived by the representation 
		$$
		\langle f^n_{\alpha,\beta},\phi_n\rangle 
		= 
		\frac{1}{2}
		\left(
		\|f^n_{\alpha,\beta}\|^2+\left\|\phi_n\right\|^2
		-
		\|f^n_{\alpha,\beta}-\phi_n\|^2
		\right) ,
		$$
		where expressions for the norms $\|f^n_{\alpha,\beta}\|$ and the distances $\|f^n_{\alpha,\beta}-\phi_n\|$ are given in Section~\ref{sec:dist} below.
		Namely, we have
		\begin{align*}
		\langle f^n_{\alpha,\beta},\phi_n\rangle
			&= \frac{4\alpha^2}{(2\sqrt{\alpha}-n)(3\sqrt{\alpha}-n)(\sqrt{\alpha}+n)}\frac{\cos\left(\frac{n\pi}{2\sqrt{\alpha}}\right)}{\sqrt{\alpha}-n}
			\quad \text{if}~ \alpha\geq n^2\geq\beta,\\
		\langle f^n_{\alpha,\beta},\phi_n\rangle
		    &=	\frac{4\beta^2}{(2\sqrt{\beta}-n)(3\sqrt{\beta}-n)(\sqrt{\beta}+n)}\frac{\cos\left(\frac{n\pi}{2\sqrt{\beta}}\right)}{\sqrt{\beta}-n}
		    \quad \text{if}~ \beta > n^2 > \alpha.
		\end{align*}
		It is not hard to observe that $\langle f^n_{\alpha,\beta},\phi_n\rangle \neq 0$ for any $n \in \mathbb{N}$.
	\end{remark}

	Using Proposition \ref{prop:scalar}, we can establish the $\omega$-linear independence of arbitrary Fu\v{c}\'ik systems stated in Proposition \ref{thm:omega}.
	\begin{proof}[Proof of Proposition \ref{thm:omega}]
		Let $F_{\alpha,\beta}=\{f^n_{\alpha,\beta}\}_{n\in\mathbb{N}_0}$ be a Fu\v{c}\'ik system.
		We need to prove that if a sequence $\{\eta_n\}_{n\in\mathbb{N}_0} \subset \mathbb{R}$ yields the strong convergence 
		\begin{equation}\label{eq:converg0}
			\lim_{k \to \infty} 
			\Big\|
			\sum_{n=0}^{k}\eta_n f^n_{\alpha,\beta}
			\Big\|
			=0,
		\end{equation}
		then we have $\eta_n=0$ for every $n\in\mathbb{N}_0$.
		
		For any fixed $m \in \mathbb{N}_0$, we consider the sum $\sum_{n=0}^{k} \eta_n \langle f^n_{\alpha,\beta}, \phi_m \rangle$ and apply the Cauchy inequality to obtain
		\begin{equation}\label{eq:converg}
			\Big| 
			\sum_{n=0}^{k} \eta_n \langle f^n_{\alpha,\beta}, \phi_m \rangle 
			\Big|
			=
			\Big| 
			\int_0^\pi \cos (mx) \sum_{n=0}^{k}\eta_n f^n_{\alpha,\beta}(x) \, \mathrm{d}x
			\Big|
			\leq
			\sqrt{\frac{\pi}{2}} \,
			\Big\|
			\sum_{n=0}^{k}\eta_n f^n_{\alpha,\beta}
			\Big\|
			.
		\end{equation}
		Thanks to \eqref{eq:converg0}, the estimate \eqref{eq:converg} implies
		\begin{equation}\label{eq:converg1}
		    \lim_{k\to\infty}
			\Big| 
			\sum_{n=0}^{k} \eta_n \langle f^n_{\alpha,\beta}, \phi_m \rangle 
			\Big|
			=0.
		\end{equation}
		By Proposition \ref{prop:scalar} - \ref{prop:scalar:I} \ref{prop:scalar:I:iii} and \ref{prop:scalar:II} \ref{prop:scalar:II:iii} -  we have $\langle f^n_{\alpha,\beta}, \phi_m \rangle = 0$ for every $n > m \geq 1$, and $\langle f^0_{\alpha,\beta}, \phi_m \rangle = 0$ for every $m\in\mathbb{N}$ since $f^0_{\alpha,\beta}=\phi_0$. Furthermore, the inequality	$\langle f^m_{\alpha,\beta}, \phi_m \rangle \neq 0$ holds for every $m\in\mathbb{N}$, see Remark \ref{rem:nn}.
		
		We start by setting $m=1$ in \eqref{eq:converg1} and deduce that $\eta_1=0$.
		Next, setting $m=2$, we deduce that $\eta_2=0$, and so on. Arguing by induction, we conclude that $\eta_n=0$ for every $n\in\mathbb{N}$. Finally, the equality $\eta_0=0$ follows directly from \eqref{eq:converg0}. This completes the proof of the $\omega$-linear independence of $F_{\alpha,\beta}$.
	\end{proof}

	\section{Distances and Proof of Theorem \ref{thm:main1}}\label{sec:dist}
	
In this section, we derive expressions for the distances $\|f^n_{\alpha,\beta}-\phi_n\|^2$ and use them to prove Theorem \ref{thm:main1}. We also provide some formulas for the norms of Fu\v{c}\'ik eigenfunctions.
	\begin{lemma}\label{lem:distances}
		Let $n \in \mathbb{N}$ and $(\alpha,\beta) \in \Gamma_n$.
		If $\alpha\geq n^2\geq\beta$, then
		\begin{equation}\label{eq:norm}
			\|f^n_{\alpha,\beta}-\phi_n\|^2 = \pi-\pi\frac{n(\sqrt{\alpha}-n)}{(2\sqrt{\alpha}-n)^2} - \frac{8\alpha^2}{(2\sqrt{\alpha}-n)(3\sqrt{\alpha}-n)(\sqrt{\alpha}+n)}\frac{\cos\left(\frac{n\pi}{2\sqrt{\alpha}}\right)}{\sqrt{\alpha}-n}.
		\end{equation}
		For $\beta > n^2 > \alpha$, we have
		\begin{equation}\label{eq:norm2}
			\|f^n_{\alpha,\beta}-\phi_n\|^2 = \pi-\pi\frac{n(\sqrt{\beta}-n)}{(2\sqrt{\beta}-n)^2} - \frac{8\beta^2}{(2\sqrt{\beta}-n)(3\sqrt{\beta}-n)(\sqrt{\beta}+n)}\frac{\cos\left(\frac{n\pi}{2\sqrt{\beta}}\right)}{\sqrt{\beta}-n}.
		\end{equation}	
	\end{lemma}
	\begin{proof}	
		Let us consider the case $\alpha\geq n^2\geq\beta$. Since $f^n_{\alpha,\beta}$ equals $\phi_n$ for $\alpha = n^2 = \beta$, we assume $\alpha>n^2>\beta$. We begin by deriving the integral $\int_0^{\pi/n} (f^n_{\alpha,\beta}-\phi_n)^2 \,\mathrm{d}x$, where the interval of integration represents a single down-slope of $f^n_{\alpha,\beta}$, i.e., the interval between a maximum and a consecutive minimum. 
		Using the pointwise definition \eqref{eq:fucikpiecewise1} and the formulas \eqref{eq:ap:3} and \eqref{eq:ap:5}, we obtain 
		\begin{align*}
			\int_0^{\pi/n} &(f^n_{\alpha,\beta}-\phi_n)^2 \,\mathrm{d}x 
			\\
			&= \int_0^{l_1/2}\left(\frac{\sqrt{\beta}}{\sqrt{\alpha}}\cos(\sqrt{\alpha}x)-\cos(nx)\right)^2	\mathrm{d}x 
			+ \int_{l_1/2}^{\pi/n}\left(\cos\left(\sqrt{\beta}\left(x-\frac{\pi}{n}\right)\right)+\cos(nx)\right)^2	\mathrm{d}x \\
			&= \frac{\beta}{\alpha}\frac12\left(x+\frac1{\sqrt{\alpha}}\sin(\sqrt{\alpha}x)\cos(\sqrt{\delta}x)\right) \bigg|_0^{l_1/2} \\
			&- 2\frac{\sqrt{\beta}}{\sqrt{\alpha}}\left(\frac{n}{n^2-\alpha}\cos(\sqrt{\alpha}x)\sin(nx)
			- \frac{\sqrt{\alpha}}{n^2-\alpha}\sin(\sqrt{\alpha}x)\cos(nx)\right) \bigg|_0^{l_1/2} \\
			&+ \frac12\left(x+\frac1{n}\sin(nx)\cos(nx)\right) \bigg|_0^{l_1/2} \\
			&+ \frac12\left(x-\frac{\pi}{n}+\frac1{\sqrt{\beta}}\sin\left(\sqrt{\beta}\left(x-\frac{\pi}{n}\right)\right)\cos\left(\sqrt{\beta}\left(x-\frac{\pi}{n}\right)\right)\right) \bigg|_{l_1/2}^{\pi/n} \\
			&+ 2 \left(\frac{n}{n^2-\beta}\cos\left(\sqrt{\beta}\left(x-\frac{\pi}{n}\right)\right)\sin(nx)
			- \frac{\sqrt{\beta}}{n^2-\beta}\sin\left(\sqrt{\beta}\left(x-\frac{\pi}{n}\right)\right)\cos(nx)\right) \bigg|_{l_1/2}^{\pi/n} \\
			&+ \frac12\left(x+\frac1{n}\sin(nx)\cos(nx)\right) \bigg|_{l_1/2}^{\pi/n}.
		\end{align*}
		Directly rearranging the resulting terms, we arrive at	
	    \begin{align*}
	        \int_0^{\pi/n} (f^n_{\alpha,\beta}-\phi_n)^2 \,\mathrm{d}x
			&= \frac{\pi}{n}+\frac12\left(\frac{\beta}{\alpha}-1\right)\frac{l_1}{2}+2\sqrt{\beta}\left(\frac1{n^2-\alpha}-\frac1{n^2-\beta}\right)\cos\left(n\frac{l_1}2\right) \\
			&= \frac{\pi}{n}-\pi\frac{\sqrt{\alpha}-n}{(2\sqrt{\alpha}-n)^2} - \frac{8\alpha^2}{n(\sqrt{\alpha}+n)(2\sqrt{\alpha}-n)(3\sqrt{\alpha}-n)}\frac{\cos\left(\frac{n\pi}{2\sqrt{\alpha}}\right)}{\sqrt{\alpha}-n}.
		\end{align*}
		Since both the Fu\v{c}\'ik eigenfunction $f^n_{\alpha,\beta}$ considered on the whole $\mathbb{R}$ and the cosine function $\phi_n$ are symmetric with respect to each extrema at $x=\frac{k\pi}{n}$, see Remark \ref{rem:sym}, we get the same formula for the integral over a single up-slope $\int_{\pi/n}^{2\pi/n} (f^n_{\alpha,\beta}-\phi_n)^2 \,\mathrm{d}x$, i.e, the interval between a minimum and a consecutive maximum. Thanks to the $\frac{2\pi}{n}$-periodicity of $f^n_{\alpha,\beta}-\phi_n$, we sum up the integrals over all subintervals $\left[\frac{k\pi}{n},\frac{(k+1)\pi}{n}\right]$, $k=0,\ldots,n-1$, to obtain the formula \eqref{eq:norm} for every $n\in\N$. 
		For the case $\beta>n^2>\alpha$, we get the similar formula \eqref{eq:norm2} due to the symmetry properties given in Remark~\ref{rem:sym}. 
	\end{proof}

	We can make use of Lemma \ref{lem:distances} to prove Theorem \ref{thm:main1}.
	\begin{proof}[Proof of Theorem \ref{thm:main1}]
		We start by providing an upper bound on $\|f^n_{\alpha,\beta}-\phi_n\|^2$ for the case $\alpha\geq n^2\geq\beta$. 
		We use the lower bound
		$$
		\cos(x) 
		=
		\sin\left(\frac{\pi}{2}-x\right)
		\geq \frac{1}{6}\left(\frac{\pi}{2}-x\right)\left(\sqrt{6}-\frac{\pi}{2}+x\right)\left(\sqrt{6}+\frac{\pi}{2}-x\right),
		\quad x \leq \frac{\pi}{2},
		$$
		which in our particular case reads as
		$$
		\cos\left(\frac{\pi}2\frac{n}{\sqrt{\alpha}}\right) 
		\geq
		\frac{\pi(\sqrt{\alpha}-n)}{48\sqrt{\alpha^3}}((2\sqrt{6}-\pi)\sqrt{\alpha}+\pi n)((2\sqrt{6}+\pi)\sqrt{\alpha}-\pi n).
		$$
		We obtain from \eqref{eq:norm} that 
		\begin{align}
			\notag
			\|f^n_{\alpha,\beta}-\phi_n\|^2
			&\leq \pi-\pi\frac{n(\sqrt{\alpha}-n)}{(2\sqrt{\alpha}-n)^2} - \pi\frac{\sqrt{\alpha}((2\sqrt{6}-\pi)\sqrt{\alpha}+\pi n)((2\sqrt{6}+\pi)\sqrt{\alpha}-\pi n)}{6(\sqrt{\alpha}+n)(2\sqrt{\alpha}-n)(3\sqrt{\alpha}-n)} \\
			\notag
			&=
			\pi \frac{2(12+\pi^2)\alpha - (30-\pi^2)\sqrt{\alpha} n - 12n^2}{6(\sqrt{\alpha}+n)(2\sqrt{\alpha}-n)^2(3\sqrt{\alpha}-n)} (\sqrt{\alpha}-n)^2\\
			\notag
			&\leq
			\pi \frac{2(12+\pi^2)\alpha - (42-\pi^2)n^2}{6(\sqrt{\alpha}+n)(3\sqrt{\alpha}-n)} \frac{(\sqrt{\alpha}-n)^2}{(2\sqrt{\alpha}-n)^2} \\
			\notag
			&=
			\pi \frac{(12+\pi^2)\left(\sqrt{\alpha} + \sqrt{\frac{42-\pi^2}{24+2\pi^2}}n\right)\left(3\sqrt{\alpha} -3\sqrt{\frac{42-\pi^2}{24+2\pi^2}}n\right)}
			{9(\sqrt{\alpha}+n)(3\sqrt{\alpha}-n)} \frac{(\sqrt{\alpha}-n)^2}{(2\sqrt{\alpha}-n)^2}\\
			\label{eq:norm:upper}
			&\leq
			\frac{(12+\pi^2)\pi}{9} \frac{(\sqrt{\alpha}-n)^2}{(2\sqrt{\alpha}-n)^2}
			\leq
			\frac{(12+\pi^2)\pi}{9}
			\frac{(\sqrt{\alpha}-n)^2}{n^2}.
		\end{align}
		In the case $\beta>n^2>\alpha$, we start with \eqref{eq:norm2} and get the same estimate with $\beta$ instead of $\alpha$.
		
		Using the obtained upper bounds and the assumption \eqref{eq:thm:main1}, we see that
		$$
		\sum_{n=1}^\infty 
		\|f^n_{\alpha,\beta}-\phi_n\|^2
		\leq 
		\frac{(12+\pi^2)\pi}{9}
		\sum_{n=1}^\infty 
		\frac{(\max(\sqrt{\alpha(n)},\sqrt{\beta(n)})-n)^2}{n^2}
		<\infty.
		$$
		Thus, the Fu\v{c}\'ik system $F_{\alpha,\beta}$ is quadratically near to the complete orthogonal system of cosine functions $\{\phi_n\}_{n\in\mathbb{N}_0}$.
		Therefore, the rescaled system $\sqrt{2/\pi} \, F_{\alpha,\beta}$ is quadratically near to the complete \textit{orthonormal} system $\{\sqrt{2/\pi} \, \phi_n\}_{n\in\mathbb{N}_0}$. 
		Together with the $\omega$-linear independence of $\sqrt{2/\pi} \, F_{\alpha,\beta}$ established in Proposition \ref{thm:omega}, this implies that $\sqrt{2/\pi} \, F_{\alpha,\beta}$ is a Riesz basis in $L^2(0,\pi)$, see, e.g.,  \cite[Theorem V-2.20]{kato}. Hence, $F_{\alpha,\beta}$ is also a Riesz basis in $L^2(0,\pi)$.
	\end{proof}
	
	\begin{remark}\label{rem:conv}
	    We notice that the assumption \eqref{eq:thm:main1} is, in fact, equivalent to the weaker assumption
	    \begin{equation}
	    \label{eq:assu}
	    \sum_{n=1}^\infty 
		\|f^n_{\alpha,\beta}-\phi_n\|^2<\infty
	    \end{equation}
	    which is required in the application of \cite[Theorem V-2.20]{kato}. Indeed, the distance $\|f^n_{\alpha,\beta}-\phi_n\|^2$ can be estimated from below as follows.
	    For $\alpha\geq n^2\geq\beta$, we get
		$$
		\cos\left(\frac{n\pi}{2\sqrt{\alpha}}\right) 
		\leq
		\frac{\pi}2-\frac{n\pi}{2\sqrt{\alpha}}=\frac{\pi}{2}\frac{\sqrt{\alpha}-n}{\sqrt{\alpha}},
		$$
		and hence
		\begin{align}
		    \notag
			\|f^n_{\alpha,\beta}-\phi_n\|^2 
			&\geq \pi-\pi\frac{n(\sqrt{\alpha}-n)}{(2\sqrt{\alpha}-n)^2} - \pi\frac{4\sqrt{\alpha}^3}{(\sqrt{\alpha}+n)(2\sqrt{\alpha}-n)(3\sqrt{\alpha}-n)} \\
			\label{eq:es_lb}
			&= \pi\frac{(4\alpha+5\sqrt{\alpha}n-2n^2)(\sqrt{\alpha}-n)^2}{(\sqrt{\alpha}+n)(2\sqrt{\alpha}-n)^2(3\sqrt{\alpha}-n)}.
		\end{align}
		Comparing \eqref{eq:es_lb} with the first term in \eqref{eq:norm:upper}, we get
		$$
		\frac{(12+\pi^2)\pi}{9} \frac{(\sqrt{\alpha}-n)^2}{(2\sqrt{\alpha}-n)^2}
		\leq
		C \cdot
		\pi\frac{(4\alpha+5\sqrt{\alpha}n-2n^2)}{(3\sqrt{\alpha}-n)(\sqrt{\alpha}+n)}\frac{(\sqrt{\alpha}-n)^2}{(2\sqrt{\alpha}-n)^2}, \quad n\in\mathbb{N},
		$$
		where $C=\dfrac{12+\pi^2}{12}$ is a uniform constant. For the case $\beta>n^2>\alpha$, we can easily derive a lower bound similar to \eqref{eq:es_lb} and obtain an estimate with the same constant $C$. 
		On the other hand, it is not hard to observe that
	    $$
		\sum_{n=1}^\infty 
    	\frac{(\max(\sqrt{\alpha(n)},\sqrt{\beta(n)})-n)^2}{(2\max(\sqrt{\alpha(n)},\sqrt{\beta(n)})-n)^2}
		<\infty
		$$
		if and only if
		$$
		\sum_{n=1}^\infty
		\frac{(\max(\sqrt{\alpha(n)},\sqrt{\beta(n)})-n)^2}{n^2} < \infty.
    	$$
    	This shows the equivalence of the assumptions \eqref{eq:thm:main1} and \eqref{eq:assu}.
	\end{remark}

	We conclude this section by providing expressions for the norms of normalized Fu\v{c}\'ik eigenfunctions which were used in Remark~\ref{rem:nn}. 
	\begin{lemma}
		Let $n \in \mathbb{N}$ and $(\alpha,\beta) \in \Gamma_n$.
		If $\alpha\geq n^2\geq\beta$, then 
		$$
		\|f^n_{\alpha,\beta}\|^2
		=
		\frac{\pi}{2} - \frac{\pi n(\sqrt{\alpha}-n)}{(2\sqrt{\alpha}-n)^2}.
		$$
		For $\beta > n^2 > \alpha$, we have
		$$
		\|f^n_{\alpha,\beta}\|^2 = \frac{\pi}{2} - \frac{\pi n(\sqrt{\beta}-n)}{(2\sqrt{\beta}-n)^2}.
		$$
	\end{lemma}

	\section{Proof of Theorem \ref{thm:main2}}\label{sec:proof}
    Let $\alpha$ and $\beta$ be the mappings of an arbitrary Fu\v{c}\'ik system $F_{\alpha,\beta}$.	
	Let us consider projections of Fu\v{c}\'ik eigenfunctions into the Hilbert space
	$$
	L^2_\perp(0,\pi)
	=
	\{
	u \in L^2(0,\pi):~ \langle u,\phi_0\rangle=0
	\},
	$$
	which is the closed linear subspace of $L^2(0,\pi)$ consisting of functions with zero mean. A complete orthogonal system in $L^2_\perp(0,\pi)$ is given by $\{\phi_n\}_{n \in \mathbb{N}}$. We define the projections of the elements of $F_{\alpha,\beta}$ into $L^2_\perp(0,\pi)$ by
	\begin{equation}\label{eq:gn}
		g_n(x) 
		:=
		f^n_{\alpha,\beta}(x) 
		- 
		\frac{\sqrt{2}}{\pi}
		\langle f^n_{\alpha,\beta},\phi_0\rangle, 
		\quad 
		n \in \mathbb{N}.
	\end{equation}
	Obviously, we have $\langle g_n,\phi_0\rangle=0$ for every $n \in \mathbb{N}$. 
	
	We start by proving that $\{g_n\}_{n \in \mathbb{N}}$ is a Riesz basis in $L^2_\perp(0,\pi)$. 
    For that purpose, we show that there exist constants $C_{n,k}$, $n,k \in \mathbb{N}$, with bounds $|C_{n,k}| \leq c_k$ satisfying $\sum_{k=1}^{\infty} c_k < 1$, and linear isometries $T_k$ of $L^2_\perp(0,\pi)$, $k \in \mathbb{N}$, such that the following representation holds:
	\begin{equation}\label{eq:duffin-g}
		g_n(x) = \cos(nx) + \sum_{k=1}^{\infty}C_{n,k} T_k \cos(nx),
		\quad n \in \mathbb{N}.
	\end{equation}
	Then, \cite[Theorem D]{duff} implies that $\{\sqrt{2/\pi} \, g_n\}_{n \in \mathbb{N}}$ shares the Riesz basis property in $L^2_\perp(0,\pi)$ with the complete orthonormal system $\{\sqrt{2/\pi} \, \phi_n\}_{n \in \mathbb{N}}$. See also \cite[Section 10]{young}.
	Clearly, $\{g_n\}_{n \in \mathbb{N}}$ is a Riesz basis in  $L^2_\perp(0,\pi)$, as well.
	
	Let us construct the linear isometries $T_k:L^2_\perp(0,\pi) \mapsto L^2_\perp(0,\pi)$, $k \in \mathbb{N}$, in such a way that
	\begin{equation}\label{eq:tk}
		T_k \cos(nx) = \cos(nkx),
		\quad
		n \in \mathbb{N}.
	\end{equation}
	We take any $h \in L^2_\perp(0,\pi)$ and define its even $2\pi$-periodic extension $h^*$ as follows:
	$$
	h^*(x) 
	=
	\left\{
	\begin{aligned}
		&h(x- 2 \pi \kappa)  &&\text{for}~~ 2 \pi\kappa \leq x \leq \pi (2\kappa+1),~~ \kappa \in \mathbb{N}_0,\\
		&h(\pi (2\kappa+1)-x) &&\text{for}~~  \pi (2\kappa+1) \leq x \leq 2 \pi (\kappa+1), ~~\kappa \in \mathbb{N}_0.
	\end{aligned}
	\right.
	$$
	For any $n \in \mathbb{N}$, we get
	$$
	n\int_0^\pi h^*(nx) \, \mathrm{d}x
	=
	\int_0^{n\pi} h^*(y) \, \mathrm{d}y
	=
	n \int_0^{\pi} h^*(y) \, \mathrm{d}y
	=0
	$$
	and hence $h^*(n\, \cdot) \in L^2_\perp(0,\pi)$.
	We define $T_k$ by
	\begin{equation*}
		T_k h(x) = h^*(kx).
	\end{equation*}
	Evidently, \eqref{eq:tk} is satisfied.	Arguing as in \cite[Lemma 8]{BM2016}, it can be shown that each $T_k$ is indeed a linear isometry of $L^2_\perp(0,\pi)$.
	In particular, $T_1$ equals the identity operator.
	
    By the property \eqref{eq:tk} of the isometries $T_k$, the claimed representation formula \eqref{eq:duffin-g} takes the form
	\begin{equation}\label{eq:duffin-g-0}
		g_n(x) = \cos(nx) + \sum_{k=1}^{\infty}C_{n,k}  \cos(nkx),
		\quad n \in \mathbb{N}.
	\end{equation}
    In order to show that such representation exists,  	
	we observe, using the pointwise definition~\eqref{eq:fucikpiecewise1},  that each $f^n_{\alpha,\beta}$ has a dilated structure in the sense of
	\begin{equation}\label{eq:f-even1}
		f_{\alpha(n),\beta(n)}^n(x) = 
		f_{\gamma_n,\frac{\gamma_n}{(2\sqrt{\gamma_n}-1)^2}}^1\left(nx\right),
	\end{equation}
	for $\alpha\geq n^2\geq\beta$, where $\left(\gamma_n,\frac{\gamma_n}{(2\sqrt{\gamma_n}-1)^2}\right)\in\Gamma_1$ is the projection of $(\alpha(n),\beta(n))\in\Gamma_n$ along the line through the origin to the first non-trivial curve of the Fu\v{c}\'ik spectrum, i.e., we have
	\begin{equation*}\label{eq:dilated-map}
	\alpha(n)=n^2\gamma_n, \qquad 	\beta(n)=\frac{n^2\gamma_n}{(2\sqrt{\gamma_n}-1)^2}.
	\end{equation*}
	In view of \eqref{eq:f-even1}, we see that $g_n$ also has a dilated structure:
	\begin{equation}\label{eq:dilated-g}
		g_n(x) 
		= 
		h_n(nx)
		:=
		f_{\gamma_n,\frac{\gamma_n}{(2\sqrt{\gamma_n}-1)^2}}^1(nx)
		-
		\frac{\sqrt{2}}{\pi}\langle f^n_{\alpha,\beta},\phi_0\rangle
	\end{equation}
	with
	\begin{equation*}
		h_n
		=
		f_{\gamma_n,\frac{\gamma_n}{(2\sqrt{\gamma_n}-1)^2}}^1
		-
			\frac{\sqrt{2}}{\pi}\langle f^n_{\alpha,\beta},\phi_0\rangle
		=
		f_{\gamma_n,\frac{\gamma_n}{(2\sqrt{\gamma_n}-1)^2}}^1
		-
			\frac{\sqrt{2}}{\pi}\langle f^1_{\gamma_n,\frac{\gamma_n}{(2\sqrt{\gamma_n}-1)^2}},\phi_0\rangle
		\in L^2_\perp(0,\pi).
	\end{equation*}
	Similarly, if $\beta > n^2 > \alpha$ for some $n\in\mathbb{N}$, we again obtain a dilated structure for $f^n_{\alpha,\beta}$:
	$$
	f_{\alpha(n),\beta(n)}^n(x) = 
		f_{\frac{\delta_n}{(2\sqrt{\delta_n}-1)^2},\delta_n}^1\left(nx\right),
	$$
	where $\delta_n > 1$ satisfies
		\begin{equation*}\label{eq:dilated-map2}
			\alpha(n)=\frac{n^2\delta_n}{(2\sqrt{\delta_n}-1)^2}, 
			\qquad 	
			\beta(n)= n^2\delta_n,
		\end{equation*}
	and we have
	\begin{equation}\label{eq:dilated-g-0}
		g_n(x) 
		= 
		\tilde{h}_n(nx)
		:=
		f_{\frac{\delta_n}{(2\sqrt{\delta_n}-1)^2},\delta_n}^1(nx)
		-
		\frac{\sqrt{2}}{\pi}\langle f^n_{\alpha,\beta},\phi_0\rangle
	\end{equation}
	with $\tilde{h}_n \in L^2_\perp(0,\pi)$.
	
	In view of \eqref{eq:dilated-g} and \eqref{eq:dilated-g-0}, we have
	\begin{equation*}\label{eq:fourier-g}
		g_n(x) 
		= 
		\sum_{k=1}^{\infty}A_{n,k}\cos(nkx),
		\quad
		n \in \mathbb{N},
	\end{equation*}
	where the constants $A_{n,k}$, $k\in\mathbb{N}$, are the Fourier coefficients of the functions $h_n$ or $\tilde{h}_n$, depending on the relation of $\alpha$ to $n^2$. 
	Setting $C_{n,1}=A_{n,1}-1$ and $C_{n,k}=A_{n,k}$ for $k \geq 2$, we get the desired representation formula \eqref{eq:duffin-g-0}.
	
	Let us show that $|C_{n,k}| \leq c_k$, where $\sum_{k=1}^{\infty} c_k < 1$.
	We derive estimates only for the case $\alpha \geq n^2 \geq \beta$, since estimates for the case $\beta > n^2 > \alpha$ are identical due to Remark~\ref{rem:sym}. Hereinafter, we write $\gamma$ instead of $\gamma_n$, for brevity. If $\gamma=1$, then  $f^n_{\alpha,\beta}=\phi_n$ with $A_{n,1}=1$ and $A_{n,k}=0$ for every $k\geq2$. 
	Assume that $\gamma>1$. We use the expressions \eqref{eq:scal:mdivn0} and \eqref{eq:scal:aeqm-mdivn-main}
	to obtain
	\begin{align}
		A_{n,k} 
		&= 
		\frac2{\pi}\int_0^{\pi}h_n(x)\cos(kx)\,\mathrm{d}x  
		=
		\label{eq:Ak}
		\frac{8 \gamma^2 (1-\sqrt{\gamma})\cos\left(\frac{k\pi}{2\sqrt{\gamma}}\right)}{\pi(k^2-\gamma)(k^2(2\sqrt{\gamma}-1)^2-\gamma)(2\sqrt{\gamma}-1)},
	\end{align}
	provided $\gamma  \neq k^2$, and
	\begin{equation}\label{eq:ak2}
		A_{n,k} 
		= 
		\frac{1}{2k (2k-1)}
	\end{equation}	
	for $\gamma = k^2$, $k\geq2$. It is not hard to deduce from \eqref{eq:Ak} that
	\begin{align*}
		A_{n,1}-1
		&=
		\frac{8 \gamma^2\cos\left(\frac{\pi}{2\sqrt{\gamma}}\right)}{\pi(\gamma-1)(2\sqrt{\gamma}-1)(3\sqrt{\gamma}-1)}-1
		\leq
		\frac{4 \sqrt{\gamma^3}}{(\sqrt{\gamma}+1)(2\sqrt{\gamma}-1)(3\sqrt{\gamma}-1)}-1 < 0,
	\end{align*}
	thanks to the upper bound
	\begin{equation}\label{eq:cos}
		\Big|\cos\left(\frac{\pi}{2\sqrt{x}}\right)\Big|
		=
		\Big|\sin\left(\frac{\pi}{2}-\frac{\pi}{2\sqrt{x}}\right)\Big|
		\leq
		\Big|\frac{\pi}{2} - \frac{\pi}{2\sqrt{x}}\Big| = \frac{\pi|\sqrt{x}-1|}{2\sqrt{x}},
		\quad x>0.
	\end{equation}
	Thus, we estimate $|C_{n,1}|$ as follows:
	\begin{align}
		\notag
		|C_{n,1}| = 1-A_1
		&\leq
		1-
		\frac{\sqrt{\gamma}((\sqrt{24}-\pi)\sqrt{\gamma} + \pi)((\sqrt{24}+\pi)\sqrt{\gamma} - \pi)}{6(\sqrt{\gamma}+1)(2\sqrt{\gamma}-1)(3\sqrt{\gamma}-1)}\\
		\label{eq:c1}
		&=
		\frac{(\sqrt{\gamma}-1)((12+\pi^2)\gamma+(18-\pi^2)\sqrt{\gamma}-6)}{6(\sqrt{\gamma}+1)(2\sqrt{\gamma}-1)(3\sqrt{\gamma}-1)},
	\end{align}
	where we used the lower bound 
	\begin{align*}
		\cos\left(\frac{\pi}{2\sqrt{\gamma}}\right)
		&\geq
		\frac{\pi(\sqrt{\gamma}-1)}{2\sqrt{\gamma}}
		-
		\frac{\pi^3(\sqrt{\gamma}-1)^3}{48\sqrt{\gamma^3}}\\
		&=
		\frac{\pi(\sqrt{\gamma}-1)((\sqrt{24}-\pi)\sqrt{\gamma} + \pi)((\sqrt{24}+\pi)\sqrt{\gamma} - \pi)}{48 \sqrt{\gamma^3}},
		\quad \gamma>1.
	\end{align*}
	The right-hand side in \eqref{eq:c1} is strictly increasing with respect to $\gamma \in (1,\infty)$, see Lemma \ref{lem:rational}.
	Therefore, we obtain the uniform bound
	\begin{equation}\label{eq:c1-main}
		|C_{n,1}| 
		\leq 
		\lim_{\gamma \to \infty}
		\frac{(\sqrt{\gamma}-1)((12+\pi^2)\gamma+(18-\pi^2)\sqrt{\gamma}-6)}{6(\sqrt{\gamma}+1)(2\sqrt{\gamma}-1)(3\sqrt{\gamma}-1)}
		=
		\frac{12+\pi^2}{36}.
	\end{equation}
	Next, let us estimate $C_{n,k}$ with $k \geq 2$.
	Assume first that $\gamma \neq k^2$. 
	Using \eqref{eq:cos}, we obtain from \eqref{eq:Ak} that
	\begin{align}
		\notag
		|C_{n,k}| = |A_{n,k}|
		&\leq
		\frac{8 \gamma^2 (\sqrt{\gamma}-1)}{\pi |k^2-\gamma| (k^2(2\sqrt{\gamma}-1)^2-\gamma)(2\sqrt{\gamma}-1)}\Big|\cos\left(\frac{k\pi}{2\sqrt{\gamma}}\right)\Big|\\
		\label{eq:ck}
		&\leq
		\frac{4 \sqrt{\gamma^3} (\sqrt{\gamma}-1)}{(k+\sqrt{\gamma}) (k^2(2\sqrt{\gamma}-1)^2-\gamma)(2\sqrt{\gamma}-1)}. 
	\end{align}
	Thanks to Lemma \ref{lem:rational}, the right-hand side in \eqref{eq:ck} is strictly increasing with respect to $\gamma \in (1,\infty)$, which yields
	\begin{equation}\label{eq:ck-main}
		|C_{n,k}| 
		\leq
		\lim_{\gamma \to \infty}
		\frac{4 \sqrt{\gamma^3} (\sqrt{\gamma}-1)}{(k+\sqrt{\gamma}) (k^2(2\sqrt{\gamma}-1)^2-\gamma)(2\sqrt{\gamma}-1)}
		=
		\frac{2}{4k^2-1}.
	\end{equation}
	Assume now that $\gamma = k^2$ for some $k\geq2$.
	It is evident from \eqref{eq:ak2} that in this case $C_{n,k}$ satisfy the same bound as \eqref{eq:ck-main}: 
	$$
	|C_{n,k}| = \frac{1}{2k(2k-1)} \leq \frac{2}{4k^2-1}.
	$$
	Let us recall that the same bounds as  \eqref{eq:c1-main} and \eqref{eq:ck-main} hold true also in the case $\beta>n^2>\alpha$, since the expressions \eqref{eq:scal:mdivn0} and \eqref{eq:scal:aeqm-mdivn-main}
	are identical to \eqref{eq:scal:mdivn10} and \eqref{aeqm-mdivn-main2}, respectively, up to the sign and the change of $\alpha$ to $\beta$.	
	
	Finally, we deduce from \eqref{eq:c1-main} and \eqref{eq:ck-main} that
	$$
	\sum_{k=1}^{\infty} c_k 
	\leq
	\frac{12+\pi^2}{36}
	+
	2 \sum_{k=2}^{\infty}\frac{1}{4k^2-1}
	=
	\frac{12+\pi^2}{36}
	+
	\frac{1}{3}
	=
	0.9408223\ldots < 1.
	$$
	Therefore, applying \cite[Theorem D]{duff}, we conclude that $\{g_n\}_{n \in \mathbb{N}}$ is a Riesz basis in $L^2_\perp(0,\pi)$.
	
	\smallskip
	
	Let us now show that $F_{\alpha,\beta}$ is complete in $L^2(0,\pi)$. Taking arbitrary $\xi \in L^2(0,\pi)$, we have $\xi - \frac{\sqrt{2}}{\pi}\langle \xi,\phi_0\rangle \in L^2_\perp(0,\pi)$. 
	Hence, there exists a sequence $\{b_k\}_{k \in \mathbb{N}} \subset \mathbb{R}$ satisfying the following property: for every $\varepsilon>0$, we can find $K \in \mathbb{N}$ such that
	$$
	\Big\|
	\xi- \frac{\sqrt{2}}{\pi}\langle \xi,\phi_0\rangle - \sum_{n=1}^k b_n g_n
	\Big\| \leq \varepsilon
	$$
	for every $k \geq K$.
	Recalling the definition \eqref{eq:gn} of $g_n$, we conclude that
	\begin{equation}\label{eq:approx-f}
	\Big\|\xi 
	- 
	\frac{2}{\pi}\Big(\langle \xi,\phi_0\rangle 
	-  
	\sum_{n=1}^k \langle f^n_{\alpha,\beta},\phi_0\rangle b_n\Big) 
	f^0_{\alpha,\beta}
	- 
	\sum_{n=1}^k b_n f^n_{\alpha,\beta}\Big\| \leq \varepsilon.
	\end{equation}
	That is, $F_{\alpha,\beta}$ is complete in $L^2(0,\pi)$, and the proof of Theorem \ref{thm:main2} is finished.
	
	\begin{remark}\label{rem:b0}
		We emphasize that 
		$$
		b_0 = b_0(k) :=
		\frac{2}{\pi}\Big(\langle \xi,\phi_0\rangle 
		-  
		\sum_{n=1}^k \langle f^n_{\alpha,\beta},\phi_0\rangle b_n\Big)
		$$
		is the only coefficient in the approximation \eqref{eq:approx-f} of $\xi$ by $F_{\alpha,\beta}$ which depends on $k$. 
		It is not hard to see from \eqref{eq:approx-f} that if $\{b_0(k)\}_{k \in \mathbb{N}}$ converges for any $\xi \in L^2(0,\pi)$, then $F_{\alpha,\beta}$ is a basis in $L^2(0,\pi)$.
		Since $\sum_{n=1}^\infty b_n g_n$ converges strongly and $\{g_n\}_{n\in\mathbb{N}}$ is a Riesz basis in $L^2_\perp(0,\pi)$, we have 
		$\sum_{n=1}^\infty b_n^2 < \infty$, see, e.g., \cite[Theorem 9]{young}. 
		On the other hand, we obtain from \eqref{eq:scalar-m0-1} and \eqref{eq:scalar-m0-2}
		that
		\begin{equation}\label{eq:conv:scal21}
			\sum_{n=1}^\infty
			\langle f^n_{\alpha,\beta},\phi_0\rangle^2
			=
			8\sum_{n=1}^\infty
			\frac{(\max(\sqrt{\alpha(n)},\sqrt{\beta(n)})-n)^2}{(2\max(\sqrt{\alpha(n)},\sqrt{\beta(n)})-n)^2}.
		\end{equation}
		In view of Remark \ref{rem:conv},  the convergence of \eqref{eq:conv:scal21} is equivalent to the assumption \eqref{eq:thm:main1} of Theorem \ref{thm:main1}.
		Thus, we conclude by
		$$
		\sum_{n=1}^\infty |\langle f^n_{\alpha,\beta},\phi_0\rangle b_n|
		\leq 
		\Big(\sum_{n=1}^\infty \langle f^n_{\alpha,\beta},\phi_0\rangle^2\Big)^{1/2}
		\Big(\sum_{n=1}^\infty b_n^2\Big)^{1/2}<\infty
		$$
		the absolute convergence of $\sum_{n=1}^\infty \langle f^n_{\alpha,\beta},\phi_0\rangle b_n$ if \eqref{eq:thm:main1} is satisfied.
		This gives the convergence of $\{b_0(k)\}_{k \in \mathbb{N}}$ and provides an alternative proof of the basisness of $F_{\alpha,\beta}$.
	\end{remark}

	\section{Biorthogonal system}\label{sec:biort}
	
	
	Each Fu\v{c}\'{i}k system $F_{\alpha,\beta}$ which is a Riesz basis in the Hilbert space $L^2(0,\pi)$ possesses a complete biorthogonal sequence $\{\psi_m\}_{m \in \mathbb{N}_0}$, i.e., a sequence with $\langle f^n_{\alpha,\beta},\psi_m\rangle = \delta_{n,m}$ for every $m,n\in\mathbb{N}_0$, that satisfies
	$$
	\sum_{m=0}^{\infty} \langle f,\psi_m\rangle^2<\infty
	$$
	for any $f\in L^2(0,\pi)$, see, e.g., \cite[Chapter I, Theorem 9]{young}. 
	The decomposition of a function $f\in L^2(0,\pi)$ with respect to the Riesz basis $F_{\alpha,\beta}$  is then given by
	$$
	f = \sum_{n=0}^{\infty} \langle f,\psi_n\rangle f^n_{\alpha,\beta}.
	$$
	We want to provide closed form expressions for the biorthogonal sequence $\{\psi_m\}_{m \in \mathbb{N}_0}$. 
	First, we consider the case $m\geq1$ and 
    use the following ansatz on the form of $\psi_m$:
	$$
	\psi_m=\sum_{k=1}^{m}C^m_k\phi_k.
	$$
	This choice guarantees that $\langle f^n_{\alpha,\beta},\psi_m\rangle=0$ for every $n>m$, see Proposition \ref{prop:scalar}.
	Let us derive the constants $C^m_k$. 
	We start by considering the scalar product of $\psi_m$ with $f^m_{\alpha,\beta}$ and obtain
	$$
	1 = \langle f^m_{\alpha,\beta},\psi_m\rangle = \sum_{k=1}^{m}C^m_k\langle f^m_{\alpha,\beta},\phi_k\rangle = C^m_m\langle f^m_{\alpha,\beta},\phi_m\rangle,
	$$
	where we used that $\langle f^m_{\alpha,\beta},\phi_k\rangle=0$ for all $k<m$, 
	see Proposition \ref{prop:scalar}.
	We know from Remark~\ref{rem:nn} that $\langle f^m_{\alpha,\beta},\phi_m\rangle\neq0$ every $m\in\mathbb{N}$, which yields  $C^m_m=\langle f^m_{\alpha,\beta},\phi_m\rangle^{-1}$. 
	Successively taking the scalar products of $\psi_m$ with $f^{m-1}_{\alpha,\beta},f^{m-2}_{\alpha,\beta},\ldots$, we obtain the remaining constants recursively as follows:
	$$
	0 = \langle f^{m-l}_{\alpha,\beta},\psi_m\rangle =  \sum_{k=m-l}^{m}C^m_k\langle f^{m-l}_{\alpha,\beta},\phi_k\rangle = C^m_{m-l}\langle f^{m-l}_{\alpha,\beta},\phi_{m-l}\rangle + \sum_{\substack{k=m-l+1 \\ (m-l)|k}}^{m}C^m_k\langle f^{m-l}_{\alpha,\beta},\phi_k\rangle,
	$$
	where $1 \leq l \leq m-1$.
	Hence, we have
	$$
	C^m_{m-l}=-\langle f^{m-l}_{\alpha,\beta},\phi_{m-l}\rangle^{-1}\sum_{\substack{k=m-l+1 \\ (m-l)|k}}^{m}C^m_k\langle f^{m-l}_{\alpha,\beta},\phi_k\rangle
	$$
	for $1\leq l\leq m-1$. By an inductive argument, we can further show that $C^m_k=0$ for all $m/k\notin\mathbb{N}$. 
	
	We summarize our observations on the form of $\psi_m$ for $m\in\N$ in the following statement.
	\begin{theorem}\label{thm:biortF}
		Let a Fu\v{c}\'{i}k system $F_{\alpha,\beta}$ be a Riesz basis in $L^2(0,\pi)$. Then for $m\geq1$ the elements $\psi_m$ of the corresponding biorthogonal system $\{\psi_m\}_{m \in \mathbb{N}_0}$ have the form
		\begin{equation}\label{eq:biort1}
		\psi_m=\sum_{k|m}C^m_k\phi_k
		\end{equation}
		with $C^m_m=\langle f^m_{\alpha,\beta},\phi_m\rangle^{-1}$ and
		$$
		C^m_k=-\langle f^k_{\alpha,\beta},\phi_k\rangle^{-1}\sum_{l\in d(k,m)} C^m_l\langle f^k_{\alpha,\beta},\phi_l\rangle
		$$
		for each $k<m$ such that $k|m$, where $d(k,m)=\{n\in\mathbb{N}:n|m\wedge k|n\wedge n>k\}$.
	\end{theorem}
	Second, we characterize the remaining element $\psi_0$.	
	Let
	$$
	g_n(x):= f^n_{\alpha,\beta}(x) 
	- 
	\frac{\sqrt{2}}{\pi}
	\langle f^n_{\alpha,\beta},\phi_0\rangle.
	$$
	We know from Theorem \ref{thm:main2} that $\{g_n\}_{n \in \mathbb{N}}$ is a Riesz basis in $L^2_\perp(0,\pi)$.
	Consequently, there is a corresponding unique biorthogonal system. 
	This system is given by the sequence $\{\psi_m\}_{m \in \mathbb{N}}$ described by Theorem \ref{thm:biortF}.
	Indeed, we notice that $\psi_m \in L^2_\perp(0,\pi)$ for $m \geq 1$ by the form \eqref{eq:biort1}, and hence
	\begin{equation*}\label{eq:biortg}
	\langle g_n, \psi_m\rangle
	=
	\langle f^n_{\alpha,\beta}, \psi_m\rangle
	-
	\frac{\sqrt{2}}{\pi}\langle f^n_{\alpha,\beta},\phi_0\rangle
	\langle 1, \psi_m\rangle
	=
	\langle f^n_{\alpha,\beta}, \psi_m\rangle
	=
	\delta_{n,m}
	\end{equation*}
	for every $m,n \in \mathbb{N}$.
	
	Taking now any $\xi \in L^2(0,\pi)$, we have $\xi - \frac{\sqrt{2}}{\pi}\langle \xi,\phi_0\rangle \in L^2_\perp(0,\pi)$. 
	Hence, for every $\varepsilon>0$ there exists $K \in \mathbb{N}$ such that
	$$
	\Big\|\xi- \frac{\sqrt{2}}{\pi}\langle \xi,\phi_0\rangle 
	- 
	\sum_{n=1}^k \langle \xi- \frac{\sqrt{2}}{\pi}\langle \xi,\phi_0\rangle , \psi_n \rangle g_n\Big\| 
	=
	\Big\|\xi- \frac{\sqrt{2}}{\pi}\langle \xi,\phi_0\rangle 
	- 
	\sum_{n=1}^k \langle \xi, \psi_n \rangle g_n\Big\| 
	\leq \varepsilon
	$$
	for every $k\geq K$.
	Recalling the definition of $g_n$, we conclude that
	\begin{align*}
		\Big\|\xi 
		- 
		\frac{\sqrt{2}}{\pi}\Big(\langle \xi,\phi_0\rangle 
		-  
		\sum_{n=1}^k \langle f^n_{\alpha,\beta},\phi_0\rangle \langle \xi, \psi_n \rangle\Big) 
		- 
		\sum_{n=1}^k \langle \xi, \psi_n \rangle f^n_{\alpha,\beta}\Big\|
		\leq \varepsilon.
	\end{align*}
	On the other hand, since we assume that $F_{\alpha,\beta}$ is a Riesz basis in $L^2(0,\pi)$, there exists $K_1 \in \mathbb{N}$ such that for any $k \geq K_1$ we have 
	$$
	\Big\|\xi 
	- 
	\frac{\sqrt{2}}{\pi}\langle \xi, \psi_0 \rangle 
	- 
	\sum_{n=1}^k \langle \xi, \psi_n \rangle f^n_{\alpha,\beta}\Big\|
	\leq \varepsilon.
	$$
	Therefore, by the triangle inequality, we get for any $k \geq \max\{K,K_1\}$ that
	
	$$
	\Big\|
	\Big(\langle \xi,\phi_0\rangle 
	-  
	\sum_{n=1}^k \langle f^n_{\alpha,\beta},\phi_0\rangle \langle \xi, \psi_n \rangle\Big) 
	-
	\langle \xi, \psi_0 \rangle
	\Big\|
	\leq 
	\frac{2 \pi \varepsilon}{\sqrt{2}}.
	$$
	Since $\xi \in L^2(0,\pi)$ is arbitrary, this establishes the weak convergence in $L^2(0,\pi)$:
	\begin{equation}\label{eq:weaklimit}
	\phi_0
	-
	\sum_{n=1}^k \langle f^n_{\alpha,\beta},\phi_0\rangle
	\psi_n
    \;
	\rightharpoonup
	\;
	\psi_0
	\quad
	\text{as}~ k \to \infty.
	\end{equation}
	
	Let us find assumptions which guarantee the strong convergence in $L^2(0,\pi)$.
	Since $\{\psi_m\}_{m \in \mathbb{N}}$ is a biorthogonal system of the Riesz basis $\{g_n\}_{n \in \mathbb{N}}$ in $L^2_\perp(0,\pi)$, $\{\psi_m\}_{m \in \mathbb{N}}$ is a Riesz basis in $L^2_\perp(0,\pi)$ by itself, see, e.g., \cite[Theorem 8]{young}. Hence, $\sum_{n=1}^\infty \langle f^n_{\alpha,\beta},\phi_0\rangle
	\psi_n$
	converges strongly in $L^2_\perp(0,\pi)$ and, thus, in $L^2(0,\pi)$ if and only if 
	$\sum_{n=1}^\infty \langle f^n_{\alpha,\beta},\phi_0\rangle^2 < \infty$,
	see \cite[Corollary 11.2]{sing} and \cite[Theorem 9]{young}.
	Therefore, we conclude from the weak convergence \eqref{eq:weaklimit} that the strong convergence
	\begin{equation*}
	\phi_0
	-
	\sum_{n=1}^k \langle f^n_{\alpha,\beta},\phi_0\rangle
	\psi_n
	\;
	\to
	\;
	\psi_0
	\quad
	\text{as}~ k \to \infty
    \end{equation*}
	in $L^2(0,\pi)$ holds if and only if $\sum_{n=1}^\infty \langle f^n_{\alpha,\beta},\phi_0\rangle^2 < \infty$.
	
	Let us summarize the observations on $\psi_0$ in the following theorem.	
	\begin{theorem}\label{thm:biortF0}
		Let a Fu\v{c}\'{i}k system $F_{\alpha,\beta}$ be a Riesz basis in $L^2(0,\pi)$. 
		Then the first element $\psi_0$ of the corresponding biorthogonal system $\{\psi_m\}_{m \in \mathbb{N}_0}$ can be characterized through the following weak convergence in $L^2(0,\pi)$:
		\begin{equation*}\label{eq:weaklimit0}
			\phi_0
			-
			\sum_{n=1}^k \langle f^n_{\alpha,\beta},\phi_0\rangle
			\psi_n
			\;
			\rightharpoonup
			\;
			\psi_0
			\quad
			\text{as}~ k \to \infty.
		\end{equation*}
		Moreover, we have
		\begin{equation*}\label{eq:stronglimit0}
		\psi_0
		=
		\phi_0
		-
		\sum_{n=1}^\infty \langle f^n_{\alpha,\beta},\phi_0\rangle
		\psi_n
		\end{equation*}
		in the sense of the strong convergence in $L^2(0,\pi)$
		if and only if $\sum_{n=1}^\infty \langle f^n_{\alpha,\beta},\phi_0\rangle^2 < \infty$.
	\end{theorem}
	
	We obtain from \eqref{eq:scalar-m0-1} and \eqref{eq:scalar-m0-2}
	that
	\begin{equation}\label{eq:conv:scal1}
	\sum_{n=1}^\infty
	\langle f^n_{\alpha,\beta},\phi_0\rangle^2
	=
	8\sum_{n=1}^\infty
	\frac{(\max\{\sqrt{\alpha(n)},\sqrt{\beta(n)}\}-n)^2}{(2\max\{\sqrt{\alpha(n)},\sqrt{\beta(n)}\}-n)^2}.
	\end{equation}
	Notice that in view of Remark \ref{rem:conv}  the convergence of \eqref{eq:conv:scal1} is equivalent to the assumption~\eqref{eq:thm:main1} of Theorem \ref{thm:main1}.

	\appendix
	\section{}\label{sec:appendix}

	In order to prove Proposition \ref{prop:scalar} and Lemma \ref{lem:distances}, we evaluate the following integrals in the general form:
	\begin{align}
		\label{eq:ap:3}
		\int\cos^2(\sqrt{\delta}(x-x_0))\,\mathrm{d}x &= \frac12\left(x-x_0+\frac1{\sqrt{\delta}}\sin(\sqrt{\delta}(x-x_0))\cos(\sqrt{\delta}(x-x_0))\right)+C,\\
		\label{eq:ap:4}
		\int \cos(n(x-x_0))\cos(nx) \,\mathrm{d}x 
		&= \frac{1}{4 n} \sin(n (2x - x_0)) + \frac{x}{2}  \cos(n x_0) + C,
	\end{align}
	for $\delta\in\{\alpha,\beta\}$ and $n\in\mathbb{N}$.
	Moreover, if $\delta \neq n^2$, then we have
	\begin{align}
		\notag
		\int \cos(\sqrt{\delta}(x-x_0))\cos(nx) \,\mathrm{d}x 
		&= \frac{n}{n^2-\delta}
		\cos(\sqrt{\delta}(x-x_0))\sin(nx) \\
		\label{eq:ap:5}
		&- \frac{\sqrt{\delta}}{n^2-\delta}
		\sin(\sqrt{\delta}(x-x_0))\cos(nx)
		+C.
	\end{align}
    We also employ the following summation formulas. Let $d \in \mathbb{R}$ be such that $\sin(d/2) \neq 0$. Then for any $a \in \mathbb{R}$ and any $n \in \mathbb{N}$, one has
	\begin{equation}\label{eq:cos-sum}
		\sum_{k=0}^{n-1} \cos(a+kd) 
		= 
		\frac{\sin(nd/2)}{\sin(d/2)}
		\cos\left(
		a+\frac{(n-1)d}{2}
		\right),
	\end{equation}
	see, e.g., \cite{knapp}.
	On the other hand, if $d \in \mathbb{R}$ is such that $\sin(d/2) = 0$, then 
	\begin{equation}\label{eq:cos-sum-2}
		\sum_{k=0}^{n-1} \cos(a+kd) 
		= 
		n \cos(a).
	\end{equation}
	Finally, we use the following standard equalities for arbitrary $a,b,c,d \in \mathbb{R}$:
	\begin{align}\label{eq:cos-sin}
		\sin(a-b)+\sin(a+b)+\sin(c-d)+\sin(c+d) &= 2\sin(a)\cos(b)+2\sin(c)\cos(d),
	\\
	\label{eq:cos-sin2}
		2 \sin(a+b) \cos(a-b)
		&=
		\sin(2 a) + \sin(2 b).
	\end{align}
	
	\medskip
	
	\begin{proof}[Proof of Proposition \ref{prop:scalar}]
	Recall that
	$$
	l_1 = \frac{\pi}{\sqrt{\alpha}},
    \qquad
    l_2 = \frac{\pi}{\sqrt{\beta}},
    \qquad
    l = l_1 + l_2 = \frac{2\pi}{n}.
	$$
	
		\ref{prop:scalar:I} 
		Let $\alpha\geq n^2\geq\beta$.
		We start with the case that $n$ is \textit{even}.
		For the calculation of the scalar product, we use the fact that the Fu\v{c}\'ik eigenfunction $f^n_{\alpha,\beta}$ is $2\pi/n$-periodic and split the integration range into $n/2$ equidistant intervals, see Remark \ref{rem:sym}.
		We use the piecewise definition~\eqref{eq:fucikpiecewise1} of $f^n_{\alpha,\beta}$ to obtain
		\begin{align}
			\notag
			\langle f^n_{\alpha,\beta},\phi_m\rangle &=
			\sum_{k=0}^{\frac{n}{2}-1} \int_{kl}^{kl+\frac{l_1}{2}} \frac{\sqrt{\beta}}{\sqrt{\alpha}}\cos(\sqrt{\alpha}(x-kl))\cos(mx) \,\mathrm{d}x \\
			\notag
			&- \sum_{k=0}^{\frac{n}{2}-1} \int_{kl+\frac{l_1}{2}}^{\left(k+\frac12\right)l+\frac{l_2}{2}}
			\cos\left(\sqrt{\beta}\left(x-\frac{2k+1}{2}l\right)\right)\cos(mx) \,\mathrm{d}x \\
			\label{eq:scal-integral}
			&+ \sum_{k=0}^{\frac{n}{2}-1} \int_{\left(k+\frac12\right)l+\frac{l_2}{2}}^{(k+1)l}
			\frac{\sqrt{\beta}}{\sqrt{\alpha}}\cos(\sqrt{\alpha}(x-(k+1)l))\cos(mx) \,\mathrm{d}x.
		\end{align}
		Assume first that $\alpha \neq m^2$ and $\beta \neq m^2$. 
		Applying the formula \eqref{eq:ap:5}, we derive
		\begin{align}
		\notag
			\langle f^n_{\alpha,\beta},\phi_m\rangle
			&= \sum_{k=0}^{\frac{n}{2}-1}
			\left[
			- \frac{\sqrt{\beta}}{m^2-\alpha}\cos\left(m\left(kl+\frac{l_1}{2}\right)\right)
			- \frac{\sqrt{\beta}}{\sqrt{\alpha}}\frac{m}{m^2-\alpha}\sin(mkl)
			\right] \\
			\notag
			&+ \sum_{k=0}^{\frac{n}{2}-1}
			\left[
			\frac{\sqrt{\beta}}{m^2-\beta}\cos\left(m\left(\left(k+\frac12\right)l+\frac{l_2}2\right)\right)
			+ \frac{\sqrt{\beta}}{m^2-\beta}\cos\left(m\left(kl+\frac{l_1}{2}\right)\right)
			\right] \\
			\notag
			&+ \sum_{k=0}^{\frac{n}{2}-1}
			\left[
			\frac{\sqrt{\beta}}{\sqrt{\alpha}}\frac{m}{m^2-\alpha}\sin(m(k+1)l)
			- \frac{\sqrt{\beta}}{m^2-\alpha}\cos\left(m\left(\left(k+\frac12\right)l+\frac{l_2}{2}\right)\right)
			\right] \\
			\label{eq:scalarnm1}
			&= \frac{(\beta-\alpha)\sqrt{\beta}}{(m^2-\alpha)(m^2-\beta)} \sum_{k=0}^{\frac{n}{2}-1}
			\left[
			\cos\left(m\left(kl+\frac{l_1}{2}\right)\right)
			+ \cos\left(m\left(kl+\frac{l_1}{2}+l_2\right)\right)
			\right].
		\end{align}
		If $m/n\in\mathbb{N}$, then we can easily apply \eqref{eq:cos-sum-2} to get
		\begin{align*}
			\langle f^n_{\alpha,\beta},\phi_m\rangle
			&= \frac{(\beta-\alpha)\sqrt{\beta}}{(m^2-\alpha)(m^2-\beta)}\frac{n}{2}
			\left[
			\cos\left(\frac{ml_1}{2}\right) + \cos\left(\frac{ml_1}{2}+ml_2\right)
			\right] \\
			&= \frac{(\beta-\alpha)\sqrt{\beta}}{(m^2-\alpha)(m^2-\beta)}\frac{n}{2}
			\left[
			\cos\left(\frac{m\pi}{n}-\frac{ml_2}{2}\right) + \cos\left(\frac{m\pi}{n}+\frac{ml_2}{2}\right)
			\right] \\
			&= \frac{(\beta-\alpha)\sqrt{\beta}\,n}{(m^2-\alpha)(m^2-\beta)}\cos\left(\frac{m\pi}{n}+\frac{ml_2}{2}\right)
			= (-1)^{m/n}\frac{(\beta-\alpha)\sqrt{\beta}\,n}{(m^2-\alpha)(m^2-\beta)}\cos\left(\frac{ml_2}{2}\right),
		\end{align*}
		where we used the symmetry of the cosine with respect to its extrema at $m\pi/n$.
		We rewrite this expression in dependence on $\alpha$ by the relation \eqref{eq:ab} and get
		\begin{equation}\label{eq:scal:mdivn1}
			\langle f^n_{\alpha,\beta},\phi_m\rangle
			= \frac{4\alpha^2n^2(n-\sqrt{\alpha})}{(m^2-\alpha)(m^2(2\sqrt{\alpha}-n)^2-n^2\alpha)(2\sqrt{\alpha}-n)}\cos\left(\frac{m\pi}{2\sqrt{\alpha}}\right).
		\end{equation}
		For the case $m/n\notin\mathbb{N}$, we use \eqref{eq:cos-sum} to obtain from \eqref{eq:scalarnm1} that
		\begin{align*}
			\langle f^n_{\alpha,\beta},\phi_m\rangle
			&= \frac{(\beta-\alpha)\sqrt{\beta}}{(m^2-\alpha)(m^2-\beta)}\frac{\sin\left(\frac{nml}{4}\right)}{\sin\left(\frac{ml}{2}\right)} \\
			&\times \left[
			\cos\left(\frac{ml_1}{2}+\frac{(n-2)ml}{4}\right)
			+ \cos\left(\frac{ml_1}{2}+\frac{(n-2)ml}{4}+ml_2\right)
			\right] \\
			&= \frac{(\beta-\alpha)\sqrt{\beta}}{(m^2-\alpha)(m^2-\beta)}\frac{\sin\left(\frac{m\pi}{2}\right)}{\sin\left(\frac{m\pi}{n}\right)}
			\left[
			\cos\left(\frac{m\pi}{2}-\frac{ml_2}{2}\right)
			+ \cos\left(\frac{m\pi}{2}+\frac{ml_2}{2}\right)
			\right]. 
		\end{align*}
		For even $m$ the term $\sin(m\pi/2)$ vanishes and for odd $m$ the sum of cosines cancels.
		Thus, if $m/n\notin\mathbb{N}$, then $\langle f^n_{\alpha,\beta},\phi_m\rangle=0$.
		
		Assume now that $\alpha = m^2$ and $\beta \neq m^2$. 
		In this case, we apply the formulas \eqref{eq:ap:4} and \eqref{eq:ap:5} to calculate the integrals in \eqref{eq:scal-integral} as follows:
		\begin{align}
			\notag
			\langle f^n_{\alpha,\beta},\phi_m\rangle
			&=
			\sum_{k=0}^{\frac{n}{2}-1}
			\left[
			\frac{\sqrt{\beta}}{4m^2} 
			\sin(m(kl+l_1))
			-
			\frac{\sqrt{\beta}}{4m^2} 
			\sin(mkl)
			+
			\frac{\sqrt{\beta} l_1}{4m} 
			\cos(mkl)
			\right]\\
			\notag
			&+
			\sum_{k=0}^{\frac{n}{2}-1}
			\left[
			\frac{\sqrt{\beta}}{m^2-\beta}
			\cos\left(m\left(\left(k+\frac{1}{2}\right)l + \frac{l_2}{2}\right)\right)
			+
			\frac{\sqrt{\beta}}{m^2-\beta}
			\cos\left(m\left(kl + \frac{l_1}{2}\right)\right)
			\right]\\
			\notag
			&+
			\sum_{k=0}^{\frac{n}{2}-1}
			\left[
			\frac{\sqrt{\beta}}{4m^2} 
			\sin(m(k+1)l)
			-
			\frac{\sqrt{\beta}}{4m^2} 
			\sin(m(kl+l_2))
			+
			\frac{\sqrt{\beta}l_1}{4m} 
			\cos(m(k+1)l)
			\right]\\
			\notag
			&=
			\frac{\sqrt{\beta}}{4m^2}
			\sum_{k=0}^{\frac{n}{2}-1}
			\left[ 
			\sin(m(kl+l_1))-\sin(m(kl+l_2))
			\right]\\
			\notag
			&+
			\frac{\sqrt{\beta}}{m^2-\beta}
			\sum_{k=0}^{\frac{n}{2}-1}
			\left[ 
			\cos\left(m\left(\left(k+\frac{1}{2}\right)l + \frac{l_2}{2}\right)\right)
			+
			\cos\left(m\left(kl + \frac{l_1}{2}\right)\right)
			\right]\\
			\label{eq:scal:aeqm-main}
			&+
			\frac{\sqrt{\beta} \pi}{4m^2}
			\sum_{k=0}^{\frac{n}{2}-1}
			\left[ 
			\cos(mkl)+\cos(m(k+1)l)
			\right].
		\end{align}
		Since $\alpha=m^2$, we have
		\begin{align}
			\notag
			\sin(m(kl+l_1)) - \sin(m(kl+l_2))
			&=
			\sin(mkl+\pi) - \sin(m(k+1)l-\pi)
			\\
			\label{eq:scal-simplif1}
			&=
			\sin(m(k+1)l)-\sin(mkl),
		\end{align}
		and hence, recalling that $l=2\pi/n$, we get
		\begin{equation}
			\label{eq:scal:aeqm0}
			\sum_{k=0}^{\frac{n}{2}-1}
			\left[ 
			\sin(m(kl+l_1))-\sin(m(kl+l_2))
			\right]
			=
			\sum_{k=0}^{\frac{n}{2}-1}
			\left[ 
			\sin(m(k+1)l)-\sin(mkl)
			\right]
			=
			0.
		\end{equation}
		In the same way, we obtain
		\begin{align}
			\notag
			\cos\left(m\left(\left(k+\frac{1}{2}\right)l + \frac{l_2}{2}\right)\right)
			+
			\cos\left(m\left(kl + \frac{l_1}{2}\right)\right)
			&=
			\cos\left(m(k+1)l - \frac{\pi}{2}\right)
			+
			\cos\left(mkl + \frac{\pi}{2}\right)
			\\
			\label{eq:scal-simplif2}
			&=
			\sin(m(k+1)l)-\sin(mkl),
		\end{align}
		and we get
		\begin{equation}
			\label{eq:scal:aeqm1}
			\sum_{k=0}^{\frac{n}{2}-1}
			\left[ 
			\cos\left(m\left(\left(k+\frac{1}{2}\right)l + \frac{l_2}{2}\right)\right)
			+
			\cos\left(m\left(kl + \frac{l_1}{2}\right)\right)
			\right]
			=0.
		\end{equation}
		Finally, we observe that
		\begin{equation*}
			\sum_{k=0}^{\frac{n}{2}-1}
			\left[ 
			\cos(mkl)+\cos(m(k+1)l)
			\right]
			=
			2 \sum_{k=0}^{\frac{n}{2}-1} \cos(mkl)
			+(-1)^m-1.
		\end{equation*}
		If $m/n \in \mathbb{N}$, then $m$ must be even since $n$ is even, and hence \eqref{eq:cos-sum-2} implies
		\begin{equation}
			\label{eq:scal:aeqm2}
			2 \sum_{k=0}^{\frac{n}{2}-1} \cos(mkl)
			+(-1)^m-1
			=
			n.
		\end{equation}
		On the other hand, if $m/n \not\in \mathbb{N}$, then 
		we apply \eqref{eq:cos-sum} to deduce that
		\begin{align}
			2 \sum_{k=0}^{\frac{n}{2}-1} \cos(mkl)
			+(-1)^m-1
			\label{eq:scal:aeqm3}
			&=
			\frac{2\sin\left(\frac{m\pi}{2}\right)}{\sin\left(\frac{m\pi}{n}\right)} \cos\left(\frac{m\pi}{2}-\frac{m\pi}{n}\right)
			+(-1)^m-1
			=0.
		\end{align}
		Thus, substituting \eqref{eq:scal:aeqm0}, \eqref{eq:scal:aeqm1}, and either \eqref{eq:scal:aeqm2} or \eqref{eq:scal:aeqm3} into \eqref{eq:scal:aeqm-main}, we conclude that 
		if $m/n \in \mathbb{N}$, then
		\begin{equation}
			\label{aeqm-mdivn-even}
			\langle f^n_{\alpha,\beta},\phi_m\rangle
			=
			\frac{\sqrt{\beta}\pi n}{4m^2}
			=
			\frac{\pi n^2 \alpha}{8m^2 (2\sqrt{\alpha}-n)},
		\end{equation}
		while if $m/n \not\in \mathbb{N}$, then $\langle f^n_{\alpha,\beta},\phi_m\rangle=0$.

		Assume finally that $\alpha \neq m^2$ and $\beta = m^2$.
		Notice that $n>m$ and $m/n\not\in\mathbb{N}$, since $\alpha \geq n^2 \geq \beta$.
		As above, we apply the formulas \eqref{eq:ap:4} and \eqref{eq:ap:5} to calculate the integrals in \eqref{eq:scal-integral} as follows:
		\begin{align}
			\notag
			\langle f^n_{\alpha,\beta},\phi_m\rangle
			&=
			\sum_{k=0}^{\frac{n}{2}-1}
			\left[
			-\frac{m}{m^2-\alpha}
			\cos\left(m\left(kl+\frac{l_1}{2}\right)\right)
			-\frac{m^2}{\sqrt{\alpha}(m^2-\alpha)}
			\sin(mkl)
			\right]\\
			\notag
			&+
			\sum_{k=0}^{\frac{n}{2}-1}
			\bigg[
			-\frac{1}{4m}
			\sin\left(m\left(\frac{2k+1}{2}l + l_2\right)\right)
			+
			\frac{1}{4m}
			\sin\left(m\left(\frac{2k-1}{2}l + l_1\right)\right)
			\\
			\notag
			&\hspace{189pt}
			-\frac{l_2}{2}
			\cos\left(m\frac{2k+1}{2}l\right)
			\bigg]
			\\
			\notag
			&+
			\sum_{k=0}^{\frac{n}{2}-1}
			\left[
			\frac{m^2}{\sqrt{\alpha} (m^2-\alpha)}
			\sin(m(k+1)l)
			-
			\frac{m}{m^2-\alpha}
			\cos\left(m\left(\left(k+\frac{1}{2}\right)l + \frac{l_2}{2}\right)\right)
			\right]\\
			\notag
			&=
			-\frac{m}{m^2-\alpha}
			\sum_{k=0}^{\frac{n}{2}-1}
			\left[
			\cos\left(m\left(kl+\frac{l_1}{2}\right)\right)
			+
			\cos\left(m\left(\left(k+\frac{1}{2}\right)l + \frac{l_2}{2}\right)\right)
			\right]\\
			\notag
			&-
			\frac{1}{4m}
			\sum_{k=0}^{\frac{n}{2}-1}
			\left[
			\sin\left(m\left(\frac{2k+1}{2}l + l_2\right)\right)
			-
			\sin\left(m\left(\frac{2k-1}{2}l + l_1\right)\right)
			\right]\\
			\label{eq:scal:beqm-main}
			&-\frac{l_2}{2}
			\sum_{k=0}^{\frac{n}{2}-1}
			\cos\left(m\frac{2k+1}{2}l\right).
		\end{align}
		Noting that $ml_2=\pi$, we see that
		\begin{align}
			\notag
			&\cos\left(m\left(kl+\frac{l_1}{2}\right)\right)
			+
			\cos\left(m\left(\left(k+\frac{1}{2}\right)l + \frac{l_2}{2}\right)\right)
			\\
			\label{eq:scal:beqm1}
			&=
			\cos\left(mkl+\frac{ml_1}{2}\right)
			-
			\cos\left(mkl+\frac{ml_1}{2}\right)=0
		\end{align}
		and
		\begin{align}
			\notag
			&\sin\left(m\left(\frac{2k+1}{2}l + l_2\right)\right)
			-
			\sin\left(m\left(\frac{2k-1}{2}l + l_1\right)\right)\\
			\label{eq:scal:beqm2}
			&=
			\sin\left(mkl - \frac{ml}{2} + ml_1 + 2ml_2\right)
			-
			\sin\left(mkl - \frac{ml}{2} + ml_1\right)
			=
			0.
		\end{align}
		Moreover, recalling $m/n \not\in \mathbb{N}$ and applying \eqref{eq:cos-sum}, we get
		\begin{equation}
			\label{eq:scal:beqm3}
			\sum_{k=0}^{\frac{n}{2}-1}
			\cos\left(m\frac{2k+1}{2}l\right)
			=
			\frac{\sin\left(\frac{m\pi}{2}\right)}{\sin\left(\frac{m\pi}{n}\right)}
			\cos\left(\frac{m\pi}{2}\right)=0.
		\end{equation}
		Substituting \eqref{eq:scal:beqm1}, \eqref{eq:scal:beqm2}, and \eqref{eq:scal:beqm3} into \eqref{eq:scal:beqm-main}, we conclude that 
		$\langle f^n_{\alpha,\beta},\phi_m\rangle = 0$.

		\medskip
	    Let us now suppose that $n$ is \textit{odd}. 
		In this case, the scalar product takes the following form (cf.\ \eqref{eq:scal-integral}):
		\begin{align}
			\notag
			\langle f^n_{\alpha,\beta},\phi_m\rangle &=
			\sum_{k=0}^{\frac{n-1}{2}-1} \int_{kl}^{kl+\frac{l_1}{2}} \frac{\sqrt{\beta}}{\sqrt{\alpha}}\cos(\sqrt{\alpha}(x-kl))\cos(mx) \,\mathrm{d}x \\
			\notag
			&- \sum_{k=0}^{\frac{n-1}{2}-1} \int_{kl+\frac{l_1}{2}}^{\left(k+\frac12\right)l+\frac{l_2}{2}}
			\cos\left(\sqrt{\beta}\left(x-\frac{2k+1}{2}l\right)\right)\cos(mx) \,\mathrm{d}x \\
			\notag
			&+ \sum_{k=0}^{\frac{n-1}{2}-1} \int_{\left(k+\frac12\right)l+\frac{l_2}{2}}^{(k+1)l}
			\frac{\sqrt{\beta}}{\sqrt{\alpha}}\cos(\sqrt{\alpha}(x-(k+1)l))\cos(mx) \,\mathrm{d}x \\
			\notag
			&+ \int_{\frac{n-1}{2}l}^{\frac{n-1}{2}l+\frac{l_1}{2}}
			\frac{\sqrt{\beta}}{\sqrt{\alpha}}\cos\left(\sqrt{\alpha}\left(x-\frac{n-1}{2}l\right)\right)\cos(mx) \,\mathrm{d}x \\
			\label{eq:scal-integral-2}
			&- \int_{\frac{n-1}{2}l+\frac{l_1}{2}}^{\pi}
			\cos\left(\sqrt{\beta}\left(x-\frac{n}{2}l\right)\right)\cos(mx) \,\mathrm{d}x.
		\end{align}
		Assume first that $\alpha \neq m^2$ and $\beta \neq m^2$. 
		Reusing most of the calculations from \eqref{eq:scalarnm1} and noting that
		\begin{align}
		\notag
		\int_{\frac{n-1}{2}l+\frac{l_1}{2}}^{\pi}
			\cos\left(\sqrt{\beta}\left(x-\frac{n}{2}l\right)\right)\cos(mx) \,\mathrm{d}x
			&=
			-
			\frac{\sqrt{\beta}}{m^2-\beta}
			\cos\left(m\left(\frac{n-1}{2}l + \frac{l_1}{2}\right)\right)
			\\
			\label{eq:lastintegr1}
			&=
		-\frac{\sqrt{\beta}}{m^2-\beta}
		\cos\left(m\left(\pi -\frac{l_2}{2}\right)\right)
		\end{align}
		by \eqref{eq:ap:5},
		we obtain
		\begin{equation}\label{eq:scalnm2}
		\langle f^n_{\alpha,\beta},\phi_m\rangle
		= \frac{(\beta-\alpha)\sqrt{\beta}}{(m^2-\alpha)(m^2-\beta)}
		\left[
		\sum_{k=0}^{\frac{n-1}{2}} \cos\left(m\left(kl+\frac{l_1}{2}\right)\right)
		+ \sum_{k=0}^{\frac{n-3}{2}} \cos\left(m\left(kl+\frac{l_1}{2}+l_2\right)\right)
		\right].
		\end{equation}
		For $m/n\in\mathbb{N}$, we use \eqref{eq:cos-sum-2} and the symmetry of the cosine with respect to its extrema at $m\pi/n$ to get
		\begin{align}
			\notag
			\langle f^n_{\alpha,\beta},\phi_m\rangle &=    
			\frac{(\beta-\alpha)\sqrt{\beta}}{(m^2-\alpha)(m^2-\beta)}
			\left[
			\frac{n+1}{2} \cos\left(\frac{ml_1}{2}\right)
			+ \frac{n-1}{2} \cos\left(\frac{ml_1}{2}+ml_2\right)
			\right] \\
			\notag
			&= \frac{(\beta-\alpha)\sqrt{\beta}}{(m^2-\alpha)(m^2-\beta)}   
			\left[
			\frac{n+1}{2} \cos\left(\frac{m\pi}{n}-\frac{ml_2}{2}\right)
			+ \frac{n-1}{2} \cos\left(\frac{m\pi}{n}+\frac{ml_2}{2}\right)
			\right] \\
			\notag
			&= \frac{(\beta-\alpha)\sqrt{\beta}\,n}{(m^2-\alpha)(m^2-\beta)} \cos\left(\frac{m\pi}{n}+\frac{ml_2}{2}\right) 
			= (-1)^{m/n}\frac{(\beta-\alpha)\sqrt{\beta}\,n}{(m^2-\alpha)(m^2-\beta)} \cos\left(\frac{ml_2}{2}\right) \\
			\label{eq:scal:mdivn2}
			&= \frac{4\alpha^2n^2(n-\sqrt{\alpha})}{(m^2-\alpha)(m^2(2\sqrt{\alpha}-n)^2-n^2\alpha)(2\sqrt{\alpha}-n)}\cos\left(\frac{m\pi}{2\sqrt{\alpha}}\right).
		\end{align}
		Notice that \eqref{eq:scal:mdivn2} coincides with \eqref{eq:scal:mdivn1}. 
		This establishes the claimed formula \eqref{eq:scal:mdivn0} and hence completes the proof of the assertion \ref{prop:scalar:I:i}.
		In the case of $m/n\notin\mathbb{N}$, we apply \eqref{eq:cos-sum} to \eqref{eq:scalnm2} and get
		\begin{align*}
			\langle f^n_{\alpha,\beta},\phi_m\rangle 
			&=
			\frac{(\beta-\alpha)\sqrt{\beta}}{(m^2-\alpha)(m^2-\beta)}
			\frac{1}{\sin\left(\frac{m\pi}{n}\right)}
			\sin\left(\frac{(n+1)ml}{4}\right)\cos\left(\frac{ml_1}{2}+\frac{(n-1)ml}{4}\right)
			\\
			&+
			\frac{(\beta-\alpha)\sqrt{\beta}}{(m^2-\alpha)(m^2-\beta)}
			\frac{1}{\sin\left(\frac{m\pi}{n}\right)}
			\sin\left(\frac{(n-1)ml}{4}\right)\cos\left(\frac{ml_1}{2}+ml_2+\frac{(n-3)ml}{4}\right).
		\end{align*}
		By making use of \eqref{eq:cos-sin} and the antisymmetry of the sine with respect to its zero at $m\pi$, we see that
		\begin{align*}
			\langle f^n_{\alpha,\beta},\phi_m\rangle &=
			\frac{(\beta-\alpha)\sqrt{\beta}}{2(m^2-\alpha)(m^2-\beta)}
			\frac{1}{\sin\left(\frac{m\pi}{n}\right)}
			\left[
			\sin\left(m\pi+\frac{ml_1}{2}\right)
			+ \sin\left(m\pi-\frac{ml_1}{2}\right)
			\right]=0.
		\end{align*}
		
		Assume now that $\alpha = m^2$ and $\beta \neq m^2$. 
		In this case, we argue in much the same way as in the derivation of \eqref{eq:scal:aeqm-main}. By applying   \eqref{eq:lastintegr1}, the integrals in \eqref{eq:scal-integral-2} are calculated as follows:
		\begin{align}
			\notag
			\langle f^n_{\alpha,\beta},\phi_m\rangle
			&=
			\frac{\sqrt{\beta}}{4m^2}
			\sum_{k=0}^{\frac{n-1}{2}-1}
			\left[ 
			\sin(m(kl+l_1))-\sin(m(kl+l_2))
			\right]\\
			\notag
			&+
			\frac{\sqrt{\beta}}{m^2-\beta}
			\sum_{k=0}^{\frac{n-1}{2}-1}
			\left[ 
			\cos\left(m\left(\left(k+\frac{1}{2}\right)l + \frac{l_2}{2}\right)\right)
			+
			\cos\left(m\left(kl + \frac{l_1}{2}\right)\right)
			\right]\\
			\notag
			&+
			\frac{\sqrt{\beta} \pi}{4m^2}
			\sum_{k=0}^{\frac{n-1}{2}-1}
			\left[ 
			\cos(mkl)+\cos(m(k+1)l)
			\right]
			+
			\frac{\sqrt{\beta}}{4m^2}
			\sin\left(\frac{ml(n-1)}{2}\right)
			\\
			\notag
			&+
			\frac{\sqrt{\beta}}{4m^2}
			\sin\left(m\left(\frac{n-1}{2}l + l_1\right)\right)
			-
			\frac{\sqrt{\beta}}{4m^2}
			\sin\left(\frac{ml(n-1)}{2}\right)
			\\
			\label{eq:scal:aeqm-main-odd}
			&+
			\frac{\sqrt{\beta}\pi}{4m^2}
			\cos\left(\frac{ml(n-1)}{2}\right)
			+
			\frac{\sqrt{\beta}}{m^2-\beta}
			\cos\left(m\left(\frac{n-1}{2}l + \frac{l_1}{2}\right)\right).
		\end{align}
		Using \eqref{eq:scal-simplif1}, we see that
		\begin{equation}
			\label{eq:scal:aeqm0odd}
			\sum_{k=0}^{\frac{n-1}{2}-1}
			\left[ 
			\sin(m(kl+l_1))-\sin(m(kl+l_2))
			\right]
			+
			\sin\left(m\left(\frac{n-1}{2}l + l_1\right)\right)
			=0.
		\end{equation}
		Similarly, with \eqref{eq:scal-simplif2}, we get
		\begin{equation}
			\sum_{k=0}^{\frac{n-1}{2}-1}
			\left[ 
			\cos\left(m\left(\left(k+\frac{1}{2}\right)l + \frac{l_2}{2}\right)\right)
			+
			\cos\left(m\left(kl + \frac{l_1}{2}\right)\right)
			\right]
			\label{eq:scal:aeqm1odd}
			+
			\cos\left(m\left(\frac{n-1}{2}l + \frac{l_1}{2}\right)\right)
			=0.
		\end{equation}
		Finally, we observe
		\begin{equation*}
			\sum_{k=0}^{\frac{n-1}{2}-1}
			\left[ 
			\cos(mkl)+\cos(m(k+1)l)
			\right]
			+
			\cos\left(\frac{ml(n-1)}{2}\right)
			=
			2 \sum_{k=0}^{\frac{n-1}{2}} \cos(mkl)
			-1.
		\end{equation*}
		If $m/n \in \mathbb{N}$, then 
		we apply \eqref{eq:cos-sum-2} to obtain
		\begin{equation}
			\label{eq:scal:aeqm2odd}
			2 \sum_{k=0}^{\frac{n-1}{2}} \cos(mkl)
			-1
			=
			n.
		\end{equation}
		On the other hand, if $m/n \not\in \mathbb{N}$, then 
		we use \eqref{eq:cos-sum} and \eqref{eq:cos-sin2} to deduce that
		\begin{equation}
			2 \sum_{k=0}^{\frac{n-1}{2}} \cos(mkl)
			-1
			=
			\label{eq:scal:aeqm3odd}
			\frac{2\sin\left(\frac{(n+1)ml}{4}\right) \cos\left(\frac{(n-1)ml}{4}\right)}{\sin\left(\frac{ml}{2}\right)}
			-1
			=
			\frac{\sin\left(\frac{m\pi}{n}\right)+\sin(m\pi)}{\sin\left(\frac{m\pi}{n}\right)} 
			-1
			=0.
		\end{equation}
		Therefore, substituting \eqref{eq:scal:aeqm0odd}, \eqref{eq:scal:aeqm1odd}, and either \eqref{eq:scal:aeqm2odd} or \eqref{eq:scal:aeqm3odd} into \eqref{eq:scal:aeqm-main-odd}, we conclude that 
		if $m/n \in \mathbb{N}$, then
		\begin{equation}
			\label{aeqm-mdivn-odd}
			\langle f^n_{\alpha,\beta},\phi_m\rangle
			=
			\frac{\sqrt{\beta}\pi n}{4m^2}
			=
			\frac{\pi n^2 \alpha}{8m^2 (2\sqrt{\alpha}-n)},
		\end{equation}
		while if $m/n \not\in \mathbb{N}$, then $\langle f^n_{\alpha,\beta},\phi_m\rangle=0$.		
		Noting that \eqref{aeqm-mdivn-odd} coincides with \eqref{aeqm-mdivn-even}, we complete the proof of the claimed equality \eqref{eq:scal:aeqm-mdivn-main}.
		That is, the proof of the assertion \ref{prop:scalar:I:ii} is finished.
		
		Assume finally that $\alpha \neq m^2$ and $\beta = m^2$.
		As in the derivation of \eqref{eq:scal:beqm-main}, we apply the formulas \eqref{eq:ap:4} and \eqref{eq:ap:5} together with \eqref{eq:scal:beqm1} and \eqref{eq:scal:beqm2} to calculate the integrals in \eqref{eq:scal-integral-2} as follows:
		\begin{align}
			\notag
			\langle f^n_{\alpha,\beta},\phi_m\rangle
			&=
			-\frac{l_2}{2}
			\sum_{k=0}^{\frac{n-1}{2}-1}
			\cos\left(m \frac{2k+1}{2} l \right)
			+
			\frac{m^2}{\sqrt{\alpha}(m^2-\alpha)}
			\sin\left(m\frac{n-1}{2}l\right)
			\\
			\notag
			&
			-\frac{m}{m^2-\alpha}
			\cos\left(m\left(\frac{n-1}{2}l + \frac{l_1}{2}\right)\right)
			-
			\frac{m^2}{\sqrt{\alpha}(m^2-\alpha)}
			\sin\left(m\frac{n-1}{2}l\right)
			\\
			\notag
			&
			-\frac{1}{4m}
			\sin\left(m\left(2\pi - \frac{nl}{2}\right)\right)
			+\frac{1}{4m}
			\sin\left(m\left(\left(\frac{n}{2}-1\right)l +l_1\right)\right)
			\\
			\label{eq:scalaneqmb}
			&-
			\frac{1}{2}\left(\pi - \frac{n-1}{2}l - \frac{l_1}{2}\right)
			\cos\left(\frac{mnl}{2}\right).
		\end{align}
		We have
		\begin{align*}
		&\cos\left(m\left(\frac{n-1}{2}l + \frac{l_1}{2}\right)\right)
		=
		\cos\left(m\pi - \frac{\pi}{2}\right) = 0, \\
		&\sin\left(m\left(2\pi - \frac{nl}{2}\right)\right)
		=0, \\
		&\sin\left(m\left(\left(\frac{n}{2}-1\right)l +l_1\right)\right)
		=
		\sin(m\pi - \pi)=0, \\
		&\frac{1}{2}\left(\pi - \frac{n-1}{2}l - \frac{l_1}{2}\right)
		\cos\left(\frac{mnl}{2}\right)
		=
		\frac{l_2}{4}\cos\left(m\pi\right)
		=
		\frac{l_2}{4}(-1)^m.
		\end{align*}
		Since $\alpha\geq n^2\geq\beta$, we get $m/n \not\in \mathbb{N}$. Hence, we can use \eqref{eq:cos-sum} to obtain
		\begin{align*}
			&-\frac{l_2}{2}
			\sum_{k=0}^{\frac{n-1}{2}-1}
			\cos\left(m \frac{2k+1}{2} l \right)
			=
			-\frac{l_2}{2}
			\sum_{k=0}^{\frac{n-1}{2}-1}
			\cos\left(\frac{m\pi}{n} +mkl \right)
			\\
			&=
			-\frac{l_2}{2}
			\frac{\sin\left(\frac{n-1}{2} \frac{m \pi}{n}\right)}{\sin\left(\frac{m\pi}{n}\right)}
			\cos\left(\frac{m\pi}{n}+\frac{n-3}{2}\frac{m\pi}{n}\right)
			=
			-\frac{l_2}{2} \frac{\sin\left(\frac{n-1}{2} \frac{m \pi}{n}\right)\cos\left(\frac{n-1}{2} \frac{m \pi}{n}\right)}{\sin\left(\frac{m\pi}{n}\right)}\\
			&=
			-\frac{l_2}{4} \frac{\sin\left(\frac{(n-1)m \pi}{n}\right)}{\sin\left(\frac{m\pi}{n}\right)}
			=
			\frac{l_2}{4}
			\frac{\sin\left(\frac{m\pi}{n}-m\pi\right)}{\sin\left(\frac{m\pi}{n}\right)}
			=
			\frac{l_2}{4}(-1)^m.
		\end{align*}
		Substituting the above equalities into \eqref{eq:scalaneqmb}, we conclude that $\langle f^n_{\alpha,\beta},\phi_m\rangle=0$.
		Finally, we have covered all cases for the assumption $m/n \not\in \mathbb{N}$ to justify the assertion \ref{prop:scalar:I:iii}.

		\ref{prop:scalar:II} 
		Let $\beta > n^2 > \alpha$. 
		The formulas \eqref{eq:scal:mdivn10} and \eqref{aeqm-mdivn-main2}, as well as the statement \ref{prop:scalar:II:iii}, can be established either in a similar way as their counterparts in the case $\alpha \geq n^2 \geq\beta$ above, or by using the relations in Remark~\ref{rem:sym}. 
		We omit details. 
	\end{proof}

	\section{}\label{sec:appendix:rational}
	
	We provide an auxiliary monotonicity result needed in the proof of Theorem~\ref{thm:main2}. 
	
	\begin{lemma}\label{lem:rational}
		Let $k \in \mathbb{N}$.
		Then the following two functions are strictly increasing with respect to $\gamma \in (1,\infty)$:
		$$
		B_1(\gamma) =\frac{(\sqrt{\gamma}-1)((12+\pi^2)\gamma+(18-\pi^2)\sqrt{\gamma}-6)}{(\sqrt{\gamma}+1)(2\sqrt{\gamma}-1)(3\sqrt{\gamma}-1)}
		$$
		and
		$$
		B_2(\gamma) =\frac{\sqrt{\gamma^3} (\sqrt{\gamma}-1)}{(k+\sqrt{\gamma}) (k^2(2\sqrt{\gamma}-1)^2-\gamma)(2\sqrt{\gamma}-1)}.
		$$
	\end{lemma}
	\begin{proof}
		We start by considering $B_1$.
		Let us make the substitution $\sqrt{\gamma}=x+1$ to obtain
		$$
		B_1(x)
		= 
		\frac{24x+(42+\pi^2)x^2+(12+\pi^2)x^3}
		{2+11x+17x^2+6x^3}.
		$$
		Since the ratios of coefficients at the powers of $x$ satisfy the chain of inequalities
		$$
		\frac{0}{2} 
		<
		\frac{24}{11}
		<
		\frac{42+\pi^2}{17}
		<
		\frac{12+\pi^2}{6},
		$$
		we can apply \cite[Theorem 4.4]{heikkala} to conclude that $B_1$ is increasing with respect to $x \in (0,\infty)$, and hence with respect to $\gamma \in (1,\infty)$.
		
		For $B_2$ we use the same substitution $\sqrt{\gamma}=x+1$. We perform a standard comparison of polynomials to see that the ratios of coefficients at the powers of $x$ satisfy the following chain of inequalities:
		\begin{align*}
			\frac{0}{(k-1)(k+1)^2}
			<
			\frac{1}{6k^3+7k^2-4k-5}
			&<
			\frac{3}{12k^3+18k^2-5k-9}
			\\
			&<
			\frac{3}{8k^3+20k^2-2k-7}
			<
			\frac{1}{8k^2-2}.
		\end{align*}	
		Thus,  by applying \cite[Theorem 4.4]{heikkala} again, we deduce that $B_2$ is increasing with respect to $\gamma \in (1,\infty)$.
	\end{proof}

\end{document}